\documentclass{amsart}
\usepackage{amssymb,amsmath,amsrefs}
\usepackage[colorlinks=true]{hyperref}
\hypersetup{urlcolor=blue, citecolor=green}
\usepackage[dvips]{graphicx}
\usepackage{color}
\usepackage{subcaption}
\usepackage{float}

\graphicspath{{Imagens/}}
\theoremstyle{plain}

\newtheorem*{conj*}{Conjecture}
\newtheorem*{cor*}{Corollary}
\newtheorem{theorem}{Theorem}[section]

\newtheorem{proposition}[theorem]{Proposition}

\newtheorem{lemma}[theorem]{Lemma}

\theoremstyle{definition}
\newtheorem*{def*}{Definition}
\newtheorem{remark}[theorem]{Remark}

\newtheorem{definition}[theorem]{Definition}

\newcommand{\C}{\mathcal C}

\newcommand{\ga} {\gamma}

\renewcommand{\epsilon}{\varepsilon}

\newcommand{\Z}{\mathbb{Z}}
\newcommand{\N}{\mathbb{N}}
\newcommand{\R}{\mathbb{R}}
\newcommand{\eps}{\varepsilon}

\newcommand{\dist}{\operatorname{\textit{d}}}

\newcommand{\inte}{\operatorname{Int}}

\newcommand{\diam}{\operatorname{diam}}

\newcommand{\espinha}{\operatorname{Spin}}

\title[Periodic points for cw-hyperbolic homeomorphisms]{Stable/unstable holonomies, density of periodic points, and transitivity for continuum-wise hyperbolic homeomorphisms}
\author{Bernardo Carvalho}
\author{Elias Rego}

\begin{document}

\begin{abstract}
We discuss different regularities on stable/unstable
holonomies of cw-hyperbolic homeomorphisms and prove that if a cw-hyperbolic homeomorphism has continuous joint stable/unstable holonomies, then it has a dense set of periodic points in its non-wandering set. For that, we prove that the hyperbolic cw-metric (introduced in \cite{ACCV3}) can be adapted to be self-similar (as in \cite{Ar}) and, in this case, continuous joint stable/unstable holonomies are pseudo-isometric. We also prove transitivity of cw-hyperbolic homeomorphisms assuming that the stable/unstable holonomies are isometric. In the case the ambient space is a surface, we prove that a cw$_F$-hyperbolic homeomorphism has continuous joint stable/unstable holonomies when every bi-asymptotic sector is regular.
\end{abstract}

\maketitle

\section{Introduction}

Transitivity, density of periodic points and sensitivity to initial conditions are exactly the conditions on the definition of chaos in the sense of R. Devaney \cite{D}. Even though sensitivity is the one that captures the essence of chaos, it is possible to prove it assuming transitivity and density of periodic points (provided the ambient space is infinite). Thus, these two dynamical properties together are enough to obtain chaos. Hyperbolicity is a central notion in the study of chaotic dynamical systems, having appeared in the works of Anosov and Smale as a source of chaotic dynamics \cite{A}, \cite{S}. It is important to understand how several features of the dynamics of hyperbolic systems are present on chaotic systems. Anosov diffeomorphisms satisfy expansivity, that is much stronger than sensitivity, have a dense set of periodic points in their non-wandering set, and are (conjectured to be) transitive, thus satisfy all three conditions for chaos.

The density of periodic points in the non-wandering set can be proved in a topological scenario, with no assumptions on the derivative, assuming only expansiveness and shadowing. Indeed, a non-wandering point gives origin to a segment of orbit that starts and returns very close to itself, and this creates a periodic pseudo-orbit that must be shadowed by a unique periodic point because of shadowing and expansiveness. Actually, part of the hyperbolic dynamics can be proved combining these two properties (see the monograph \cite{AH}) and because of that, systems satisfying expansiveness and shadowing are usually called Topologically Hyperbolic.

The existence of periodic points assuming only the shadowing property (and without any form of expansiveness) has been discussed before in the literature (see \cite{KP} and the references therein). It is proved there that surface homeomorphisms satisfying the shadowing property have periodic points in every $\varepsilon$-transitive class. The density of periodic points in the non-wandering set is not discussed there and neither is the case of higher-dimensional manifolds or even more general spaces such as Peano continua (which are compact, connected and locally connected metric spaces).
An intermediate scenario could be discussed assuming shadowing and a weaker form of expansiveness, such as the continuum-wise expansiveness defined by Kato in \cite{Ka} which we now define. 

\begin{definition}
A homeomorphism $f$ of a compact metric space $X$ is \emph{continuum-wise expansive} if there exists $c>0$ such that $W^u_c(x)\cap W^s_c(x)$ is totally disconnected for every $x\in X$, where $W^s_c(x)$ and $W^u_c(x)$ denote the classical $c$-stable/unstable sets of $x$. The number $c>0$ is called a cw-expansive constant of $f$.
\end{definition}

We could try to prove the density of periodic points for cw-expansive homeomorphisms satisfying the shadowing property, but the proof explained above for the expansive case does not work. These systems can admit even a cantor set of distinct points shadowing a given pseudo-orbit, so a periodic pseudo-orbit can be shadowed by non-periodic points. Also, one problem that arise is that cw-expansive homeomorphisms satisfying the shadowing property could not satisfy a few of the classical results of the hyperbolic dynamics, such as the Spectral Decomposition Theorem (see \cite{CC} for an example with an infinite number of distinct chain-recurrent classes). This problem seems to be avoided with the introduction of continuum-wise hyperbolicity in \cite{ACCV3}, which we now define.

\begin{definition} \label{cwhyp}
A homeomorphism $f$ of a compact metric space $(X,d)$ satisfies the \emph{cw-local-product-structure} if for each $\eps>0$ there exists $\delta>0$ such that $$C^s_\epsilon(x)\cap C^u_\epsilon(y)\neq\emptyset \,\,\,\,\,\, \text{whenever} \,\,\,\,\,\, d(x,y)<\delta,$$ where $C^s_\epsilon(x)$ and $C^u_\epsilon(x)$ denote the connected component of $x$ on the sets $W^s_\epsilon(x)$ and $W^u_\epsilon(x)$, respectively. The cw-expansive homeomorphisms satisfying the cw-local-product-structure are called \emph{continuum-wise hyperbolic}.
\end{definition}

This notion was introduced in the context of the theory beyond topological hyperbolicity (see \cite{ACCV2}, \cite{ACCV3}, \cite{ACCV}, \cite{CC}, \cite{CC2}). For this special class of cw-expansive homeomorphisms satisfying the shadowing property, some of the classical results of the hyperbolic theory could be proved, such as a Shadowing Lemma in the form of the L-shadowing property, the Spectral Decomposition Theorem, and a hyperbolic cw-metric that sees hyperbolicity on local stable/unstable continua. Despite that, the density of periodic points in the non-wandering set was not proved in \cite{ACCV3} for cw-hyperbolic homeomorphisms. The lack of expansiveness and the difference between unique shadowing and L-shadowing could not go unnoticed. Indeed, in \cite{ACCV} there is an example of a topologically mixing homeomorphism satisfying the L-shadowing property but without periodic points. Thus, one could expect that cw-hyperbolic homeomorphisms could have a chain-recurrent class without periodic points. 

The difference between L-shadowing and cw-hyperbolicity is also important to be observed. It is proved in \cite{ACCV} that a homeomorphism satisfies the L-shadowing property if, and only if, it satisfies the shadowing property and the assymptotic local product structure: for each $\epsilon>0$ there is $\delta>0$ such that $$\dist(x,y)<\delta \,\,\,\,\,\, \text{implies} \,\,\,\,\,\, V^s_\epsilon(x)\cap V^u_\epsilon(y)\neq\emptyset,$$ where $$V^s_\epsilon(x)=W^s_\epsilon(x)\cap W^s(x) \,\,\,\,\,\, \text{and} \,\,\,\,\,\, V^u_\epsilon(y)=W^u_\epsilon(x)\cap W^u(x).$$ Since $C^s_\epsilon(x)\subset V^s_\epsilon(x)$ and $C^u_\epsilon(x)\subset V^u_\epsilon(x)$ (see \cite{Ka}) the cw-local-product-structure is stronger than the assymptotic local product structure (in the case of cw-expansive homeomorphisms). Also, cw-expansiveness implies the existence of a hyperbolic cw-metric (see \cite{ACCV3}), that is analogous to the hyperbolic metric of Fathi \cite{Fa} on the case of expansive homeomorphisms but for cw-expansive homeomorphisms. This is a feature of cw-hyperbolicity that is not necessarily present on systems with L-shadowing and will be explored in this article extensively. It will be especially important to define and discuss distinct forms of regularity of stable/unstable holonomies for cw-hyperbolic homeomorphisms.

It is known that topologically hyperbolic homeomorphisms have a well defined local product structure $[,]$ that is continuous (see Lemma 2 and Proposition 3 in \cite{Hi}) but in the case of cw-hyperbolic homeomorphisms the local stable/unstable holonomies seem to be a bit different since their images can contain more than one point (see Definition \ref{holonomies}) and we do not know whether they are continuous or not. The main result of this paper proves density of periodic points for cw-hyperbolic homeomorphisms assuming a joint continuity on stable/unstable holonomies (see Definition \ref{cw-cont}).

\begin{theorem}\label{Theorem A}
If a $cw$-hyperbolic homeomorphism has continuous joint stable/ unstable holonomies, then it has a dense set of periodic points in its non-wandering set. If local stable/unstable holonomies of a cw-hyperbolic homeomorphism are isometric, then it is transitive.
\end{theorem}

In the proof of this result, we prove that the cw-metric can be adapted to be self-similar (see Theorem \ref{existenceself}) and that, in this case, continuity of stable/unstable holonomies imply that they are pseudo-isometric (see Theorem \ref{pseudocwN}). This step is based in \cite{Ar} where similar conclusions were obtained for topologically hyperbolic homeomorphisms using the hyperbolic metric of Fathi. The hypothesis of joint continuity will be proved in the case of cw$_F$-hyperbolic homeomorphisms of surfaces assuming regularity of all bi-asymptotic sectors (see Definition \ref{cwF} and Figure \ref{fig:setorantigo}). This is based in \cite{ACS} where the structure of local stable/unstable continua and bi-asymptotic sectors of cw-hyperbolic homeomorphisms is explored. The following is proved:

\begin{theorem}\label{cwfsurface}
If a cw$_F$-hyperbolic surface homeomorphism has only regular bi-asymptotic sectors, then it has continuous joint stable/unstable holonomies and, hence, its set of periodic points is dense in its non-wandering set.
\end{theorem}

The hypothesis of regularity of all bi-asymptotic sectors will be proved for all cw$_F$-hyperbolic homeomorphisms admiting only a finite number of spines (see Proposition \ref{allregular}). The paper is organized as follows: in Section 2 we define the local stable/unstable holonomies for cw-hyperbolic homeomorphisms and prove the density of periodic points in the non-wandering set assuming they are pseudo-isometric; in Section 3 we construct a self-similar hyperbolic cw-metric and prove that joint continuity of local stable/unstable holonomies imply they are jointly pseudo-isometric using the self-similar cw-metric; in Section 4 we restrict to the study of cw-hyperbolicity on surfaces and prove the hypothesis of joint continuity using cw$_F$-hyperbolicity and regularity on bi-asymptotic sectors; and in Section 5 we prove transitivity of cw-hyperbolic homeomorphisms assuming isometric local stable/unstable holonomies.

\vspace{+0.6cm}

\section{Cw-hyperbolic stable/unstable holonomies 
and density of periodic points}

In this section, we will prove the density of periodic points in the non-wandering set of a $cw$-hyperbolic homeomorphism $f\colon X\to X$, assuming first that joint stable/unstable holonomies are pseudo-isometric. We need to define properly the local stable/unstable holonomies for cw-hyperbolic homeomorphisms, which can be quite different from the case of topologically hyperbolic homeomorphisms in \cite{Ar}. Since local stable/unstable continua can intersect in more than one point in the cw-hyperbolic scenario, the image of local stable/unstable holonomy maps will be sets in $\mathcal{K}(X)$, the space of compact subsets of $X$ endowed with the Hausdorff metric $d^H$. 

\begin{definition}[Local stable and local unstable holonomies]\label{holonomies}
Let $c>0$ be a cw-expansive constant of $f$, $\eps=\frac{c}{2}$, and choose $\delta'\in(0,\eps)$, given by the $cw$-local-product-structure, such that 
$$C^u_\epsilon(x)\cap C^s_\epsilon(y)\neq\emptyset \,\,\,\,\,\, \text{whenever} \,\,\,\,\,\, d(x,y)<\delta'.$$ 
Thus, if $\delta=\frac{\delta'}{2}$, then

for each pair $(x,y)\in X\times X$, with $d(x,y)<\delta$, we can define $\pi^s_{x,y}\colon C^s_{\delta}(x)\to \mathcal{K}(C^s_{\eps}(y))$ by 
$$\pi^s_{x,y}(z)=C_{\eps}^u(z)\cap C^s_{\eps}(y).$$
Note that $\pi^s_{x,y}(z)\neq\emptyset$ since $z\in C^s_{\delta}(x)$ implies 
$$d(z,y)\leq d(z,x)+d(x,y)\leq \delta+\delta=\delta',$$ which, in turn, implies
$$C^u_\epsilon(z)\cap C^s_\epsilon(y)\neq\emptyset.$$ The unstable holonomy map $\pi^u_{x,y}\colon C^u_{\delta}(x)\to \mathcal{K}(C^u_{\eps}(y))$ is defined similarly by
$$\pi^u_{x,y}(z)=C_{\eps}^s(z)\cap C^u_{\eps}(y).$$

\end{definition}

\begin{remark}\label{pseudo-anosov}
We note that on the example of the pseudo-Anosov diffeomorphism of the sphere $\mathbb{S}^2$, which is the main example explored in \cite{ACCV3} of a cw-hyperbolic homeomorphism that is not topologically hyperbolic, the sets $\pi^s_{x,y}(z)$ and $\pi^u_{x,y}(z)$ consist of two distinct points if $x$ and $y$ are sufficiently close to the singularities of the stable/unstable foliations (see Figure \ref{setor}).
\end{remark}

To define pseudo-isometric local stable/unstable holonomies, we need to recall the cw-metric introduced in \cite{ACCV3}. Let $\mathcal{C}(X)$ denote the space of all subcontinua of $X$ and define
$$E=\{(p,q,C): C\in \C(X),\, p,q\in C\}.$$
For $p,q\in C$ denote $C_{(p,q)}=(p,q,C)$.
The notation $C_{(p,q)}$ implies that $p,q\in C$ and that $C\in\C(X)$.
Define
$$f(C_{(p,q)})=f(C)_{(f(p),f(q))}.$$
and consider the sets
\[
\C^s_\eps(X)=\{C\in\C(X);\diam(f^n(C))\leq\eps\, \text{ for every }\, n\geq 0\} \,\,\,\,\,\, \text{and}
\]
\[
\C^u_\eps(X)=\{C\in\C(X);\diam(f^{-n}(C))\leq\eps\, \text{ for every }\, n\geq 0\},
\]
where $\diam(A)$ denotes the diameter of the set $A$. These sets contain exactly the $\eps$-stable and $\eps$-unstable continua of $f$, respectively.

\begin{theorem}[Hyperbolic $cw$-metric-\cite{ACCV3}]
\label{teoCwHyp}
If $f\colon X\to X$ is a cw-expansive homeomorphism of a compact metric space $X$, then there is a function $D\colon E\to\R$ satisfying the following conditions.
\begin{enumerate}
\item Metric properties:
\vspace{+0.2cm}
\begin{enumerate}
 \item $D(C_{(p,q)})\geq 0$ with equality if, and only if, $C$ is a singleton,\vspace{+0.1cm}
 \item $D(C_{(p,q)})=D(C_{(q,p)})$,\vspace{+0.1cm}
 \item $D([A\cup B]_{(a,c)})\leq D(A_{(a,b)})+D(B_{(b,c)})$, $a\in A, b\in A\cap B, c\in B$.
\end{enumerate}
\vspace{+0.2cm}
\item Hyperbolicity: there exist constants $\lambda\in(0,1)$ and $\varepsilon>0$ satisfying
\vspace{+0.2cm}
	\begin{enumerate}
	\item if $C\in\C^s_\eps(X)$ then $D(f^n(C_{(p,q)}))\leq 4\lambda^nD(C_{(p,q)})$ for every $n\geq 0$,\vspace{+0.1cm}
	\item if $C\in\C^u_\eps(X)$ then $D(f^{-n}(C_{(p,q)}))\leq 4\lambda^nD(C_{(p,q)})$ for every $n\geq 0$.
	\end{enumerate}
\vspace{+0.2cm}
\item Compatibility: for each $\delta>0$ there is $\gamma>0$ such that
\vspace{+0.2cm}
\begin{enumerate}
\item if $\diam(C)<\gamma$, then $D(C_{(p,q)})<\delta$\,\, for every $p,q\in C$,\vspace{+0.1cm}
\item if there exist $p,q\in C$ such that $D(C_{(p,q)})<\gamma$, then $\diam(C)<\delta$.
\end{enumerate}
\end{enumerate}
\end{theorem}

The hyperbolic cw-metric was introduced to obtain hyperbolic contract rates on local stable/unstable continua for cw-expansive homeomorphisms. It was inspired by the hyperbolic metric of Fathi \cite{Fa} in the case of expansive homeomorphisms. One important difference is that in the cw-expansive scenario we do not have a metric on the space, but a function on the space of continua with two marked points $E$. The cw-metric is the tool used to obtain some of the hyperbolic properties for cw-hyperbolic homeomorphisms in \cite{ACCV3} and we will use it in this paper to discuss regularity properties of local stable/unstable holonomies.
Let $\eps$ and $\delta$ be as in the definition of local stable and local unstable holonomies $$\pi^s_{x,y}\colon C^s_{\delta}(x)\to \mathcal{K}(C^s_{\eps}(y)) \,\,\,\,\,\, \text{and} \,\,\,\,\,\, \pi^u_{x,y}\colon C^u_{\delta}(x)\to \mathcal{K}(C^u_{\eps}(y)).$$

\begin{definition}[Pseudo-isometric local stable and local unstable holonomies]
We say that $\pi^s_{x,y}$ is a \emph{pseudo-isometric local stable holonomy} if for each $\eta>0$ there is $\ga>0$ satisfying:
 if $C$ is a subcontinuum of $C^s_{\delta}(x)$, $p,q\in C$, $D(C_{(p,q)})\leq \ga$, and $q^*\in \pi^s_{x,y}(q)$, then there exists a subcontinuum $C^*$ of $C^s_{\eps}(y)$ containing $q^*$ and there exists $p^*\in \pi^s_{x,y}(p)\cap C^*$ such that 
 \begin{equation*}
 \left|\frac{D(C^*_{(p^*,q^*)})}{D(C_{(p,q)})}-1\right|\leq \eta.
 \end{equation*}
Analogously, we can define what means $\pi^u_{x,y}$ to be a \emph{pseudo-isometric local unstable holonomy}. 
\end{definition}

In the proof of Theorem \ref{Theorem A}, we will need to ensure the existence of rectangles of stable and unstable continua with arbitrarily small sides, but the local stable and local unstable holonomies as defined above are not enough for that. Indeed, if $D(C_{(p,q)})$ is very small then the pseudo-isometric $\pi^s_{x,y}$ ensures that $$D(C^*_{(p^*,q^*)})\leq(1+\eta)D(C_{(p,q)})$$ and, hence, $D(C^*_{(p^*,q^*)})$ is also very small; however, $\pi^s_{x,y}$ intersects $\eps$-unstable continua of points in $C$ with $C^*$, and $\eps$ can be much bigger than the numbers $D(C_{(p,q)})$ and $D(C^*_{(p^*,q^*)})$. Thus, the rectangle formed by the $\eps$-stable continua $C$ and $C^*$ and the $\eps$-unstable continua $C^u_{\eps}(p)$ and $C^u_{\eps}(q)$ has two very small sides but can have two big sides (when compared to the others). Even in the case where $p$ and $p^*$ are very close inside $C^u_{\eps}(p)$, that is, there is continua $C'\subset C^u_{\eps}(p)$ containing $p$ and $p^*$ with very small $D(C'_{(p,p^*)})$, the local stable holonomy $\pi^s_{x,y}$ does not ensure that the fourth side of the rectangle is small (see Figure \ref{Pseudo-iso}).

\begin{figure}[h]
    \centering
    \includegraphics[scale=0.6]{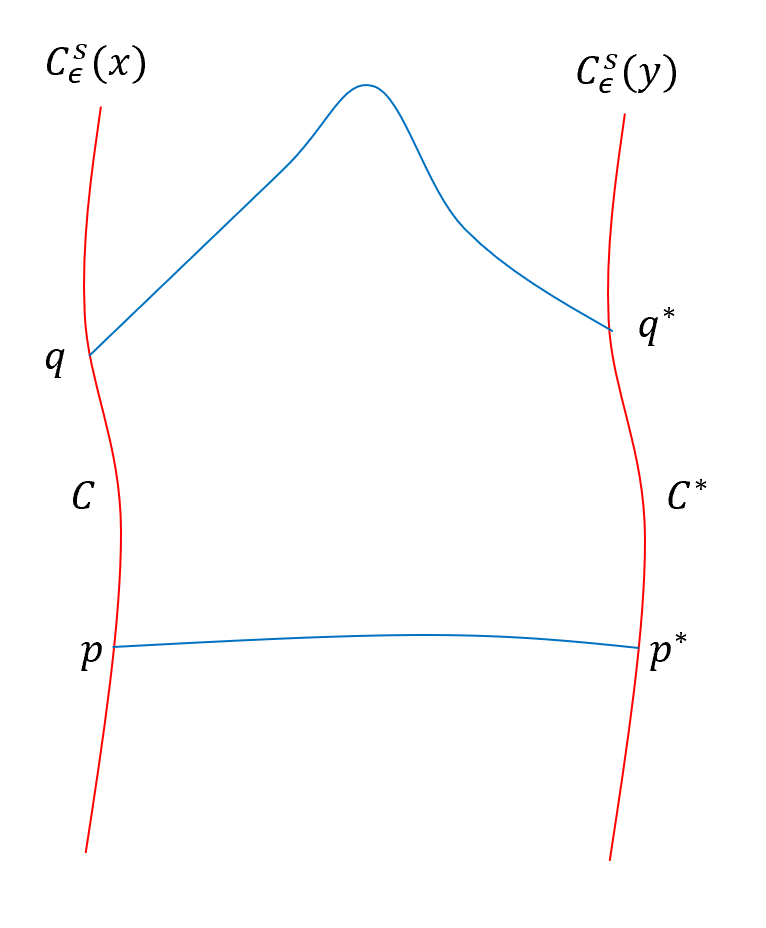}
    \caption{Pseudo-Isometric Holonomies}
    \label{Pseudo-iso}
\end{figure}

This cannot be fixed using a local unstable holonomy since it could leave a different side of the rectangle big. We remark that this does not happen in the case of expansive homeomorphisms, since the intersections between the small and the big sides of these rectangles would be two distinct points on the same dynamical ball, but it is something important to observe in the case of cw-hyperbolic homeomorphisms. This motivates the following definition (see Figure \ref{jointly pseudo-iso}):

\begin{definition}[Pseudo-isometric joint stable/unstable holonomies]\label{pseudo}
We say that $f$ has \emph{pseudo-isometric joint stable/unstable holonomies} if for each $\eta>0$ there exists $\ga>0$ satisfying:
if $d(x,y)<\delta$, $C$ is a subcontinuum of $C^s_{\eps}(x)$, $p,q\in C$, $D(C_{(p,q)})\leq \ga$, $p^*\in \pi^s_{x,y}(p)$, $C'$ is a subcontinuum of $C^u_{\eps}(q)$ containing $p$ and $p^*$, and $D(C'_{(p,p^*)})\leq\ga$, then there exist subcontinua $C^*$ of $C^s_{\eps}(y)$ containing $p^*$ and $C^{**}$ of $C^u_{\eps}(q)$ containing $q$, and there exists $q^*\in C^{*}\cap C^{**}$ such that 
$$\left|\frac{D(C^*_{(p^*,q^*)})}{D(C_{(p,q)})}-1\right|\leq \eta \,\,\,\,\,\, \text{and} \,\,\,\,\,\, \left|\frac{D(C^{**}_{(p,p^*)})}{D(C'_{(q,q^*)})}-1\right|\leq \eta.$$
\end{definition}

\begin{figure}[h]
    \centering
    \includegraphics[scale=0.6]{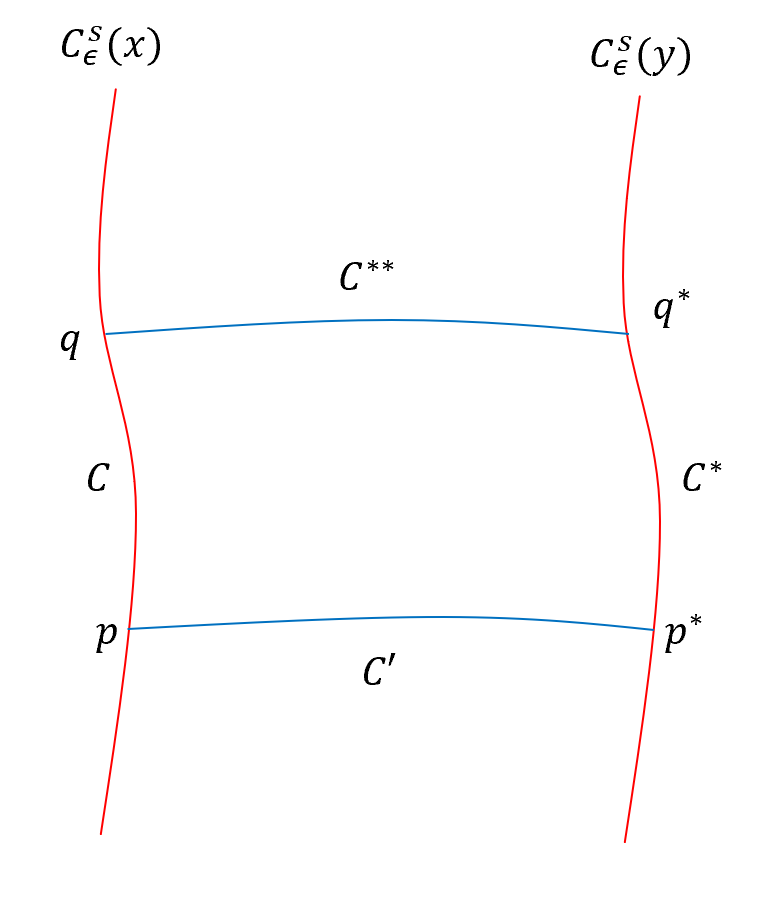}
    \caption{Jointly pseudo-isometric Holonomies}
    \label{jointly pseudo-iso}
\end{figure}

\begin{remark}\label{isom}
Note that if $f$ has pseudo-isometric local stable/unstable holonomies, then, in particular, every local stable holonomy $\pi^s_{x,y}$ is pseudo-isometric and every local unstable holonomy $\pi^u_{x,y}$ is pseudo-isometric.
\end{remark}

Now we state and prove the main result of this section. Recall that a periodic point $p\in X$ satisfies $f^k(p)=p$ for some $k\in\N$, and that $p$ is non-wandering if for each neighborhood $U$ of $p$ there exists $k\in\N$ such that $f^k(U)\cap U\neq\emptyset$. The set of periodic points of $f$ will be denoted by $Per(f)$ and the set of non-wandering points will be denoted by $\Omega(f)$.

\begin{theorem}\label{Katok}
If a cw-hyperbolic homeomorphism has pseudo isometric joint stable/unstable holonomies, then its set of periodic points is dense in its non-wandering set.
\end{theorem}

The proof is based in a paper of Katok \cite{K} which proves that positive topological entropy for a $C^{1+\alpha}$ diffeomorphism of a compact two-dimensional manifold implies the existence of hyperbolic horseshoes. An important step of the proof (see the \emph{Main Lemma} in Section 3 of \cite{K}) is the approximation of regular recurrent points by periodic points. The proof we exhibit for Theorem \ref{Katok} is inspired in the proof of this Main Lemma of Katok. Before proving this result, we prove the following elementary calculus lemma that will be used in the proof.

\begin{lemma}\label{leminha}
If $a>1$ and $0<b<1$, then for each $\eps>0$ there exists $k\in\N$ satisfying:
$$\sum_{n=0}^{\infty} a^nb^{kn}\leq\eps.$$
\end{lemma}
\begin{proof}
Recall from elementary  calculus that we can estimate $$\sum_{n=0}^{\infty} a^nb^{kn}\leq \int_0^{\infty}a^xb^{kx}dx$$
and note that $a^xb^{kx}=e^{x\log(a)+kx\log(b)}$.
Thus, 
\begin{eqnarray*}
\int_0^{\infty}a^xb^{kx}dx&=&\int_0^{\infty}e^{x\log(a)+kx\log(b)}dx\\
&=&\lim_{m\to\infty}\left[\frac{e^{m(\log(a)+k\log(b))}}{\log(a)+k\log(b)}-\frac{1}{\log(a)+k\log(b)}\right].
\end{eqnarray*}
For each $\eps>0$, choose $k>0$ such that  
$$\log(a)+k\log(b)<0 \,\,\,\,\,\, \text{and} \,\,\,\,\,\, -\frac{1}{\log(a)+k\log(b)}\leq\eps.$$
This implies that
$$\lim_{m\to\infty}e^{m(\log(a)+k\log(b))}=0$$
and, hence,
\begin{eqnarray*}
\sum_{n=0}^{\infty} a^nb^{kn}\leq\int_0^{\infty}a^xb^{kx}dx= -\frac{1}{\log(a)+k\log(b)}\leq\eps.
\end{eqnarray*}
\end{proof}

\begin{proof}[Proof of Theorem \ref{Katok}] 
Let $f\colon X\to X$ be a $cw$-hyperbolic homeomorphism and $p$ be a non-wandering point of $f$. If $p$ is periodic, then it is obviously in the closure of the set of periodic points of $f$. Thus, we can assume that $p$ is not periodic. Note that, if there is a sequence $(y_k)_{k\in\N}\subset X$ and $n\in\N$ such that 
$$y_k\to p \,\,\,\,\,\, \text{and} \,\,\,\,\,\, f^n(y_k)\to p,$$ then the continuity of $f$ assures that $p=f^n(p)$, that is, $p$ is a periodic point. Since we are assuming that $p$ is not a periodic point, $p$ cannot be accumulated in this way. Thus, for each $\alpha>0$ and $n\in\N$ we can choose $y\in X$ and $k\in\N$ such that 
$$d(y,p)<\alpha, \,\,\,\,\,\, d(f^k(y),p)<\alpha \,\,\,\,\,\, \text{and} \,\,\,\,\,\, k\geq n.$$
For each $\alpha>0$, we will prove the existence of a periodic point $q\in X$ such that $d(q,p)<\alpha$. Let $c\in(0,\frac{\alpha}{2})$ be a cw-expansive constant of $f$ and choose $\eps\in(0,\frac{c}{2})$, given by the compatibility of $D$ and $\diam$, such that
$$D(C_{(x,y)})\leq c \,\,\,\,\,\, \text{whenever} \,\,\,\,\,\, \diam(C)\leq2\eps.$$
Choose $\delta'\in(0,\eps)$, given by the $cw$-local-product-structure, such that 
$$C^u_\epsilon(x)\cap C^s_\epsilon(y)\neq\emptyset \,\,\,\,\,\, \text{whenever} \,\,\,\,\,\, d(x,y)<\delta',$$ 
and let $\delta=\frac{\delta'}{2}$. Thus, the holonomy maps $$\pi^s_{x,y}\colon C^s_{\delta}(x)\to \mathcal{K}(C^s_{\eps}(y)) \,\,\,\,\,\, \text{and} \,\,\,\,\,\, \pi^u_{x,y}\colon C^u_{\delta}(x)\to \mathcal{K}(C^u_{\eps}(y))$$ are well defined when $d(x,y)<\delta$ and are pseudo-isometric (see Remark \ref{isom}). 

Now choose $\ga\in(0,\frac{\delta}{2})$, as in Definition \ref{pseudo} of pseudo-isometric joint stable/unstable holonomies for $\eta=\delta$, and choose $\beta\in(0,\frac{\ga}{3})$, given by the compatibility between $D$ and $\diam$, such that
$$\diam(C)\leq\ga \,\,\,\,\,\, \text{whenever} \,\,\,\,\,\, D(C_{(x,y)})\leq3\beta$$ for some $x,y\in C$. Let $k_0\in\N$ be such that $4\lambda^{k_0}c\leq\beta$ and 
$$\sum_{n=0}^{\infty}[(1+\delta)^24]^n\lambda^{nk_0}\leq\beta.$$
This last inequality is ensured by Lemma \ref{leminha}.
Choose $y\in X$ and $k\geq k_0$ such that 
$$d(y,p)<\frac{\delta}{2} \,\,\,\,\,\, \text{and} \,\,\,\,\,\, d(f^k(y),p)<\frac{\delta}{2}.$$
Thus, $d(y,f^k(y))<\delta$ and, hence, there exists
$$z\in C^u_{\eps}(f^k(y))\cap C^s_{\eps}(y).$$
The hyperbolicity of $D$ ensures that
\begin{eqnarray*}
D(f^k(C^s_{\eps}(y)_{(y,z)}))&\leq&4\lambda^kD(C^s_{\eps}(y)_{(y,z)})\\
&\leq&4\lambda^kc\\
&\leq&4\lambda^{k_0}c\\
&\leq&\beta
\end{eqnarray*}
and, hence, $$\diam(f^k(C^s_{\eps}(y)))\leq\delta \,\,\,\,\,\, \text{and} \,\,\,\,\,\, f^k(C^s_{\eps}(y))\subset C^s_{\delta}(f^k(y)).$$ 
Since $D(f^k(C^s_{\eps}(y)_{(y,z)}))\leq\ga$, the pseudo-isometric $\pi^s_{f^k(y),y}$ assures that for 
$$f^k(z),f^k(y)\in f^k(C^s_{\eps}(y)) \,\,\,\,\,\, \text{and} \,\,\,\,\,\, z\in\pi^s_{f^k(y),y}(f^k(y)),$$ there exist a subcontinuum $C$ of $C^s_{\eps}(y)$ containing $z$ and $$y_1'\in \pi^s_{f^k(y),y}(f^k(z))\cap C$$ such that \begin{equation*}
 \left|\frac{D(C_{(z,y_1')})}{D(f^k(C^s_{\eps}(y)_{(y,z)}))}-1\right|\leq \delta.
 \end{equation*}
This implies that
\begin{equation*}
 D(C_{(z,y_1')})\leq(1+\delta)D(f^k(C^s_{\eps}(y)_{(y,z)}))\leq(1+\delta)4\lambda^kc.
 \end{equation*}
Let $y_1=f^{-k}(y_1')$ and see $C$ and $y_1$ in Figure \ref{arg1}.

\begin{figure}[h]
    \centering
    \includegraphics[scale=0.6]{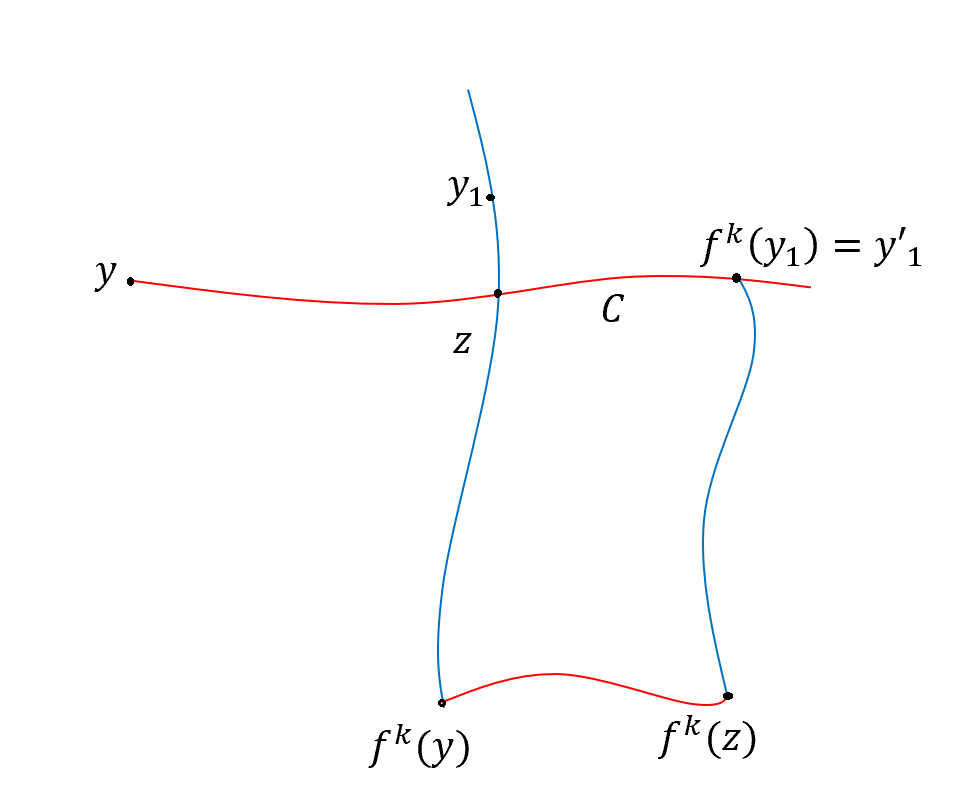}
    \caption{Construction of $y_1$}
    \label{arg1}
\end{figure}

The hyperbolicity of $D$ also ensures that
$$D(f^{-k}(C^u_{\eps}(f^k(z))_{(f^k(z),f^k(y_1))}))\leq4\lambda^kc\leq\beta$$ and, hence,
$$\diam(f^{-k}(C^u_{\eps}(f^k(z))))\leq\delta \,\,\,\,\,\, \text{and} \,\,\,\,\,\, f^{-k}(C^u_{\eps}(f^{k}(z)))\subset C^u_{\delta}(z).$$

Let $C'=f^{-k}(C^u_{\eps}(f^k(z)))$ and note that the set
$$F:=C\cup C'$$ is a continuum, since it is a union of two continua with a non-empty intersection:
$$z\in C\cap C'.$$
Now we use the metric properties and hyperbolicity of $D$ to obtain 
\begin{eqnarray*}
D(F_{(f^k(y_1),y_1)})&\leq& D(C_{(f^k(y_1),z)})+D(C'_{(z,y_1)})\\
&\leq& (1+\delta)4\lambda^kc+4\lambda^kc\\
&\leq& (2+\delta)4\lambda^kc.
\end{eqnarray*}
In particular, $$D(F_{(f^k(y_1),y_1)})\leq3\beta, \,\,\,\,\,\, \diam(F)\leq\ga \,\,\,\,\,\, \text{and} \,\,\,\,\,\, d(y_1,f^k(y_1))\leq\delta,$$ so we can repeat the argument above replacing $y$ and $f^k(y)$ by $y_1$ and $f^k(y_1)$. However, intersecting $C^s_{\eps}(y_1)$ and $C^u_{\eps}(f^k(y_1))$ will not be enough to prove the thesis of this theorem and we will need to use the pseudo-isometric joint stable/unstable holonomies, as in Definition \ref{pseudo}. 
For 
$$f^k(y_1),z\in C\subset C^s_{\delta}(z) \,\,\,\,\,\, \text{with} \,\,\,\,\,\, D(C_{(f^k(y_1),z)})\leq\ga$$ 
and 
$$y_1\in\pi^s_{z,y_1}(z)\cap C' \,\,\,\,\,\, \text{with} \,\,\,\,\,\, D(C'_{(z,y_1)})\leq\ga,$$ 
there exist subcontinua $C^*$ of $C^s_{\eps}(y_1)$ containing $y_1$ and $C^{**}$ of $C^u_{\eps}(f^k(y_1))$ containing $f^k(y_1)$, and there exists $z_1\in C^{*}\cap C^{**}$ such that 
$$\left|\frac{D(C^*_{(z_1,y_1)})}{D(C_{(f^k(y_1),z)})}-1\right|\leq \delta \,\,\,\,\,\, \text{and} \,\,\,\,\,\, \left|\frac{D(C^{**}_{(f^k(y_1),z_1)})}{D(C'_{(z,y_1)})}-1\right|\leq \delta.$$
This implies that $$D(C^{*}_{(y_1,y_1'')})\leq(1+\delta)^24\lambda^kc \,\,\,\,\,\, \text{and} \,\,\,\,\,\, D(C^{**}_{(f^k(y_1),z_1)})\leq(1+\delta)4\lambda^kc.$$ 
Repeating this argument with $y_1$ and $f^k(y_1)$ instead of $y$ and $f^k(y)$, we obtain $\eps$-stable continua $C_2$, $C_2^*$ and $\eps$-unstable continua $C_2'$, $C_2^{**}$ forming a rectangle, with 
$$y_2\in C_2'\cap C_2^{*}, \,\,\,\,\,\, f^k(y_2)\in C_2\cap C_2^{**}, \,\,\,\,\,\, z_1\in C_2\cap C_2' \,\,\,\,\,\, \text{and} \,\,\,\,\,\,  z_2\in C_2^{*}\cap C_2^{**},$$ such that 
$$D(C_{2{(f^k(y_2),z_1)}})\leq(1+\delta)^34^2\lambda^{2k}c, \,\,\,\,\,\, D(C^*_{2(z_2,y_2)})\leq(1+\delta)^44^2\lambda^{2k}c,$$
$$D(C'_{2{(z_1,y_2)}})\leq(1+\delta)^24^2\lambda^{2k}c \,\,\,\,\,\, \text{and} \,\,\,\,\,\, D(C^{**}_{2{(f^k(y_2),z_2)}})\leq(1+\delta)^34^2\lambda^{2k}c.$$ Note that $C_2$ is obtained iterating $C^*$ by $f^k$ and applying a stable/unstable holonomy, $C_2^*$ is obtained from $C_2$ applying a stable/unstable holonomy, $C_2'$ is obtained from $C^{**}$ applying a stable/unstable holonomy and iterating by $f^{-k}$, and $C_2^{**}$ is obtained from $C_2'$ applying a stable/unstable holonomy. Applying the same argument over and over we can construct, inductively, for each $n\in\N$, $\eps$-stable continua $C_n$, $C_n^*$ and $\eps$-unstable continua $C_n'$, $C_n^{**}$ forming a rectangle, with 
$$y_n\in C_n'\cap C_n^{*}, \,\,\,\,\,\, f^k(y_n)\in C_n\cap C_n^{**}, \,\,\,\,\,\, z_{n-1}\in C_n\cap C_n' \,\,\,\,\,\, \text{and} \,\,\,\,\,\,  z_n\in C_n^{*}\cap C_n^{**},$$ such that 
$$D(C_{n{(f^k(y_n),z_{n-1})}})\leq(1+\delta)^{2n-1}4^{n}\lambda^{nk}c, \,\,\,\,\,\, D(C^*_{n{(z_n,y_n)}})\leq(1+\delta)^{2n}4^{n}\lambda^{nk}c,$$
$$D(C'_{n{(z_{n-1},y_n)}})\leq(1+\delta)^{2(n-1)}4^n\lambda^{nk}c \,\,\,\,\,\, \text{and} \,\,\,\,\,\, D(C^{**}_{n{(f^k(y_n),z_n)}})\leq(1+\delta)^{2n-1}4^n\lambda^{nk}c,$$
see  Figure \ref{arg2}.
\begin{figure}[h]
    \centering
    \includegraphics[scale=0.6]{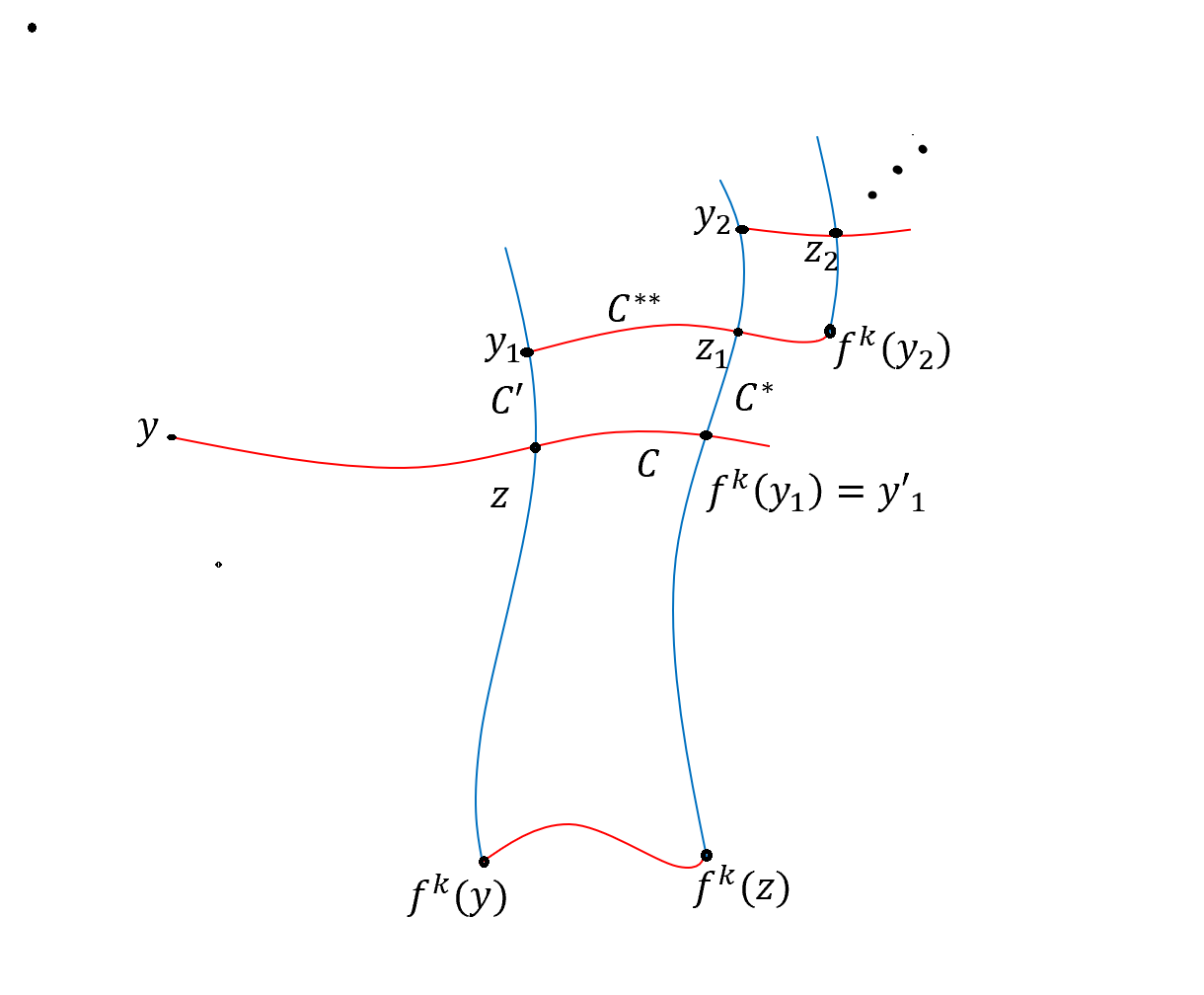}
    \caption{Construction of $y_n$.}
    \label{arg2}
\end{figure}
Let $F_n=C_n\cup C_n'$ and note that $F_n$ is a continuum since $z_{n-1}\in C_n\cap C_n'$. Using the metric properties of $D$ we obtain
\begin{eqnarray*}
D(F_{n{(f^k(y_n),y_n)}})&\leq& D(C_{n{(f^k(y_n),z_{n-1})}}) + D(C'_{n{(z_{n-1},y_n)}})\\
&\leq& (1+\delta)^{2n-1}4^{n}\lambda^{nk}c + (1+\delta)^{2(n-1)}4^n\lambda^{nk}c\\
&<&(1+\delta)^{2n}4^{n}\lambda^{nk}c\\
&=&[(1+\delta)^24]^n\lambda^{nk}c.
\end{eqnarray*}
The choice of $k\geq k_0$ and Lemma \ref{leminha} ensure that
$$D(F_{n{(f^k(y_n),y_n)}})\to0 \,\,\,\,\,\, \text{when} \,\,\,\,\,\, n\to\infty,$$
and the compatibility between $D$ and $\diam$ ensures that $$d(y_n,f^k(y_n))\to0 \,\,\,\,\,\, \text{when} \,\,\,\,\,\, n\to\infty.$$
Thus, if $q\in X$ is an accumulation point of $(y_n)_{n\in\N}$, then it is also an accumulation point of $(f^k(y_n))_{n\in\N}$, and the continuity of $f$ ensures that $q=f^k(q)$, that is, $q$ is a periodic point. We still need to prove that $d(q,p)<\alpha$. We can estimate $d(q,p)$ as follows:
\begin{eqnarray*}
d(q,p)&\leq& d(q,z)+d(z,y)+d(y,p)\\
&\leq& d(q,z)+\eps+\delta\\
&<& d(q,z)+\frac{2\alpha}{3},
\end{eqnarray*}
while $d(q,z)$ can be estimated in the following way: let $$H=\bigcup_{n\in\N}C_n^*\cup C_n',$$
where $C_1^*=C^*$ and $C_1'=C'$, note that $q,z\in\overline{H}$ since $q$ is accumulated by $y_n\in C_n^*\cap C_n'$ and $z\in C'$, and that $H$ is connected because 
$$z_{n-1}\in C_n'\cap C_{n-1}^* \,\,\,\,\,\, \text{for every} \,\,\,\,\,\, n\in\N,$$ so $\overline{H}$ is a continuum. Thus,
\begin{eqnarray*}
D(H_{(q,z)})&\leq&\sum_{n=1}^{\infty}\left[D(C^*_{n{(z_n,y_n)}})+D(C'_{n{(y_n,z_{n-1})}})\right]\\
&\leq&\sum_{n=1}^{\infty}\left[(1+\delta)^{2n}4^{n}\lambda^{nk}c+(1+\delta)^{2(n-1)}4^n\lambda^{nk}c\right]\\
&<&\sum_{n=1}^{\infty}\left[2(1+\delta)^{2n}4^{n}\lambda^{nk}c\right]\\
&\leq&2c\beta.
\end{eqnarray*}
The choice of $\beta$ ensures that
$$\diam(H)\leq\ga \,\,\,\,\,\, \text{and} \,\,\,\,\,\, d(q,z)\leq\ga<\frac{\alpha}{3}.$$ This proves that $d(q,p)<\alpha$ and finishes the proof.

 \end{proof}
 
\vspace{+0.6cm}

\section{Continuous and pseudo-isometric joint stable/unstable holonomies}\label{existself}

To prove that joint stable/unstable holonomies are pseudo isometric, we adapt the techniques in \cite{Ar} to the case of cw-hyperbolic homeomorphisms. The main idea there is to obtain a self-similar hyperbolic metric for an expansive homeomorphism and prove that in this metric, local stable/unstable holonomies are pseudo-isometric. We begin this section proving the existence of a self-similar hyperbolic $cw$-metric for $cw$-expansive homeomorphisms. We adapt the construction of the cw-metric introduced in \cite{ACCV3} and obtain a distinct function satisfying (1), (2) and (3) in Theorem \ref{teoCwHyp} and the additional property of self-similarity. We call a function $D\colon E\to\mathbb{R}$ satisfying (1), (2) and (3) a hyperbolic cw-metric.

\begin{definition}
A hyperbolic $cw$-metric $D$ is said to be self-similar if 
there exist $\xi>0$ and $\lambda>1$ such that if $D(C_{(p,q)})\leq \xi $, then
 \begin{equation}
     \max\{D(f(C_{(p,q)})),D(f^{-1}(C_{(p,q)}))\}=\lambda D(C_{(p,q)}).
 \end{equation}
  The constant $\xi$ is called the cw-expansiveness constant of $D$ and $\lambda$ is called the expansion constant of $D$.
\end{definition}
Later in this section we will prove consequences of the self-similar property for a hyperbolic cw-metric. Now, we prove its existence.

\begin{theorem}\label{existenceself} Every $cw$-expansive homeomorphism of a Peano continuum admits a hyperbolic self-similar $cw$-metric.
\end{theorem}

We start by recovering the construction of the hyperbolic $cw$-metric performed in \cite{ACCV} and then we make some modifications in order to obtain the self-similarity. Let $f$ be a $cw$-expansive homeomorphism with cw-expansiveness constant $c>0$. By the results in \cite{Ka}, there exists $m\in\N$ such that if $C\in \mathcal{C}(M)$ satisfies $diam(C)>\frac{c}{2}$, then 
\begin{equation*}
    \max_{|n|\leq m}\{diam(f^n(C))\}>c.  
\end{equation*}
For each $C\in \mathcal{C}(M)$, let 
\begin{equation}\label{rho}
    N(C)=\begin{cases}
      \min\{|n|;diam(f^n(C))>c\}, \textrm{ if C is not  a singleton}\\
      \infty, \textrm{ if C is a singleton}
     \end{cases}
\end{equation}
Let $\alpha=2^{\frac{1}{m}}$ and define $\rho:\mathcal{C}(M)\to \mathbb{R}$ by $$\rho(C)=\alpha^{-N(C)}.$$
 Now we define the function $P\colon E\to \mathbb{R}$ by 
 \begin{equation}\label{D}
     P(C_{(p,q)})=\inf \sum_{i=1}^n \rho(A^i_{(a_{i-1},a_i)}),
 \end{equation}
where the infimum is taken over all $n\geq 1$, $a_0=p$, $a_n=q$,  $A_1,\dots,A_n\in \mathcal{C}(M)$ and $C=\cup_{i=1}^n A_i$. The previous function is exactly the function introduced as the hyperbolic $cw$-metric of $f$ in \cite{ACCV}. To follow the proof in \cite{Ar} we first need to modify the function $P$ as is done by Fathi in \cite{Fa}. Choose $n_0\in\N$ such that $k=\frac{\alpha^{n_0}}{4}>1$ and let $\lambda=k^{\frac{1}{n_0}}$. Define the function $D'\colon E\to\mathbb{R}$ by
\begin{equation}
    D'(C_{(p,q)})=\max_{ |i|\leq n_0-1}\left\{\frac{P(f^i(C)_{(f^i(p),f^i(q))})}{\lambda^{|i|}}\right\}
\end{equation}
It is straightforward to verify that $D'$ is still a cw-metric. The reason this modification is done is to obtain the following estimative:
\begin{proposition}\label{Fathi'}
There exists $\delta>0$ such that if $D'(C_{(p,q)})\leq \delta $, then $$\max\{D'(f(C_{(p,q)})),D'(f^{-1}(C_{(p,q)}))\}\geq \lambda D'(C_{(p,q)}).$$
\end{proposition}
Before proving this result we recall one property of $\rho$ and $P$ that will be important in the proof. By the definition of $\rho$ we can see that $\rho(C)\leq 1$ for every $C\in\mathcal{C}(X)$, and if $$\max_{|i|\leq n-1}\rho(f^i(C))\leq \alpha^{-1},$$
then  
\begin{equation}\label{rho2}
      \max\{\rho(f^n(C)),\rho(f^{-n}(C))\}\geq \alpha^n\rho(C).
\end{equation}
Using the following inequalities proved in \cite{ACCV}:
$$P(C_{(p,q)})\leq \rho(C)\leq 4P(C_{(p,q)}),$$
a similar result can be obtained for $P$:
if $$\max_{|i|\leq n-1}P(f^i(C))\leq\frac{\alpha^{-1}}{4},$$
then  
\begin{equation}\label{P}
\max\{P(f^n(C)),P(f^{-n}(C))\}\geq\frac{\alpha^n}{4}P(C).
\end{equation}
Similar inequalities are obtained in the proof of Theorem 5.1 in \cite{Fa}.

\begin{proof}[Proof of Proposition \ref{Fathi'}]
Let $\delta=\frac{1}{4\alpha\lambda^{n_0-1}}$ and assume that $C_{(p,q)}\in E$ satisfies $D'(C_{(p,q)})\leq\delta$. This implies that $$\max_{|i|\leq n_0-1}P(f^i(C_{(p,q)}))\leq\frac{\alpha^{-1}}{4}$$ and, hence, inequality (\ref{P}) ensures that
\begin{equation}\label{P2}
\max\{P(f^{n_0}(C)),P(f^{-n_0}(C))\}\geq\frac{\alpha^{n_0}}{4}P(C).
\end{equation}
The definition of $D'$ ensures that the following inequality holds:
\begin{equation}
    \max\left\{D'(f(C_{(p,q)})),D'(f^{-1}(C_{(p,q)}))\right\}\geq \max_{0<|i|\leq n_0}\left\{\frac{P(f^i(C_{(p,q)}))}{\lambda^{|i|-1}}\right\}.
\end{equation}
    Let
    \begin{equation}\label{A}
        A=\max_{0<|i|<n_0}\lambda\left\{\frac{P(f^i(C_{(p,q)}))}{\lambda^{|i|}}\right\}.
    \end{equation}
  and
    \begin{equation}\label{B}
      B=\frac{\max\left\{P(f^{n_0}(C_{(p,q)})),P(f^{-n_0}(C_{(p,q)}))\right\}}{\lambda^{n_0-1}}.
  \end{equation}
   Inequality (\ref{P2}) ensures that 
  \begin{equation}
      B\geq \lambda P(C_{(p,q)}),
  \end{equation}
 and, hence,
 $$\max\{D'(f(C_{(p,q)})),D'(f^{-1}(C_{(p,q)})\}\geq \lambda D'(C_{(p,q)}).$$

  \end{proof}

\begin{remark}\label{upperbound}
    Note that $D'(C_{(p,q)})\leq 1$ for every $C_{(p,q)}\in E$.
\end{remark}

We finally construct a self-similar hyperbolic $cw$-metric.

  \begin{proof}[Proof of Theorem \ref{existenceself}]
  
Define $D:E\to \mathbb{R}$ by 
$$  D(C_{(p,q)})=\sup_{ i\in \mathbb{Z}}\left\{\frac{D'(f^i(C_{(p,q)}))}{\lambda^{|i|}}\right\} $$
First, note that by letting $i=0$ on the right side of above equality we obtain $D\geq D'$. This and Proposition \ref{Fathi'} are enough to prove the self-similarity. Indeed, note that
$$\max\{D(f(C_{(p,q)})),D(f^{-1}(C_{(p,q)}))\}=\max_{i\in \Z}\frac{D'(f^{i}(C_{(p,q)}))}{\lambda^{|i\pm1|}}$$
and that
$$\max_{i\in \Z}\frac{D'(f^{i}(C_{(p,q)}))}{\lambda^{|i|-1}}\\
  =\lambda D(C_{(p,q)}).$$
We prove that if $D(C_{(p,q)})\leq\delta$, then
\begin{equation}\label{max}
       \max_{i\in \Z}\frac{D'(f^{i}(C_{(p,q)}))}{\lambda^{|i\pm1|}}=\max_{i\in \Z}\frac{D'(f^{i}(C_{(p,q)}))}{\lambda^{|i|-1}}.
    \end{equation} 
Since $\min\{|i+1|,|i-1|\}\geq |i|-1$ and $\lambda>1$, it follows that
\begin{equation*}
       \max_{i\in \Z}\frac{D'(f^{i}(C_{(p,q)}))}{\lambda^{|i\pm1|}}\leq\max_{i\in \Z}\frac{D'(f^{i}(C_{(p,q)}))}{\lambda^{|i|-1}}.
    \end{equation*} 
If $D(C_{(p,q)})\leq\delta$, then $D'(C_{(p,q)})\leq D(C_{(p,q)})\leq\delta$, and Proposition \ref{Fathi'} ensures the following inequalities
\begin{equation*}
       \max_{i\in \Z}\frac{D'(f^{i}(C_{(p,q)}))}{\lambda^{|i\pm1|}}\geq\max_{i= \pm 1}\frac{D'(f^{i}(C_{(p,q)}))}{\lambda^{|i\pm1|}}\geq \lambda D'(C_{(p,q)}).
    \end{equation*}
This proves that the maximum in (\ref{max}) occur in a coordinate different from 0. But since $\min\{|i+1|,|i-1|\}$ and $|i|-1$ differ only when $i=0$, the equality (\ref{max}) is proved.

This proves that if $D(C_{(p,q)})\leq \delta$, then 
$$\max\{D(f(C_{(p,q)})),D(f^{-1}(C_{(p,q)}))\}=\lambda D(C_{(p,q)}).$$

To prove that $D$ is a cw-metric, we need to prove compatibility between $D$ and $D'$, that in turn, implies compatibility between $D$ and $\diam$. Items (1) and (2) of the definition of a cw-metric follow directly from the definition of $D$.

By Remark \ref{upperbound} we have that 
\begin{equation}\label{comp1}
    \frac{D'(f^i(C_{p,q}))}{\lambda^{|i|}}\leq \frac{1}{\lambda^{|i|}},
\end{equation}
for every $i\in \mathbb{Z}$.
For each $\eps>0$, choose $j_0>0$ such that \begin{equation}\label{comp2}
\frac{1}{\lambda^{j_0}}\leq \eps.
\end{equation}
Since $diam$ and $D'$ are compatible, we can find  $\delta>0$ such that if $diam(C)\leq \delta$, then $D'(C_{(p,q)})\leq \eps$. On the other hand, by the continuity of $f$ one can find $\eta>0$ such that if  $diam(C)\leq \eta$, then  $diam(f^i(C))\leq \delta $, for $-j_0\leq i\leq j_0$. Next we use  the compatibility between $D'$ and $diam$ again to find $\xi>0$ such that if $D'(C_{p,q})\leq \xi$, then $diam(C)\leq \eta$.

Now suppose $D'(C_{(p,q)})\leq \xi$. By the previous discussion, we have that $diam(C)\leq \eta$ and therefore, $$\sup_{-j_0\leq i \leq j_0} \{diam(f^i(C))\}\leq \delta.$$
Moreover, the choice of $\delta$ implies $$\sup_{-j_0\leq i \leq j_0} \{D'(f^i(C))\}\leq \eps.$$
In addition, since $\lambda>1$, we have

\begin{equation}\label{comp3}
\sup_{-j_0\leq i \leq j_0}\left\{ \frac{D'(f^i(C))}{\lambda^{|i|}}\right\}\leq \eps. 
\end{equation}

Finally, combining the equations (\ref{comp1}), (\ref{comp2}) and (\ref{comp3}) we obtain that $$D(C_{(p,q)})=\sup_{i\in \mathbb{Z}}\frac{D'(f^i(C_{(p,q)})}{\lambda^|i|}\leq \eps$$ and the proof is complete.

\end{proof}

From now on we will always assume that $D$ is a self-similar hyperbolic cw-metric. The second step in this section is the proof of pseudo-isometric local stable/unstable holonomies, assuming they are jointly continuous. In \cite{Ar} the assumption of continuity is hidden in the proof of pseudo-isometry, since for topologically hyperbolic homeomorphisms they are always continuous. We still do not know how to prove continuity of stable/unstable holonomies for cw-hyperbolic homeomorphisms, so continuity is an assumption for us in this section. Also, there is the difference of joint continuity/pseudo-isometry explained above that does not exist in the expansive case. 

We begin by defining joint continuity based in Definition \ref{pseudo} of pseudo-isometric stable/unstable holonomies.
\begin{definition}\label{cw-cont}[Continuous joint stable/unstable holonomies]. We say that $f$ has \emph{continuous joint stable/unstable holonomies} if for each $\eta>0$ there exists $\ga>0$ satisfying:
if $d(x,y)<\delta$, $C$ is a subcontinuum of $C^s_{\eps}(x)$, $p,q\in C$, $D(C_{(p,q)})\leq \ga$, $p^*\in \pi^u_{x,y}(p)$, $C'$ is a subcontinuum of $C^u_{\eps}(q)$ containing $p$ and $p^*$, and $D(C'_{(p,p^*)})\leq\ga$, then there exist subcontinua $C^*$ of $C^s_{\eps}(y)$ containing $p^*$ and $C^{**}$ of $C^u_{\eps}(q)$ containing $q$, and there exists $q^*\in C^{*}\cap C^{**}$ such that 
$$|D(C^*_{(p^*,q^*)})|\leq \eta \,\,\,\,\,\, \text{and} \,\,\,\,\,\, |D(C^{**}_{(p,p^*)})|\leq \eta.$$

\end{definition}

This is the same as in Figure \ref{jointly pseudo-iso} assuming that two sides of the rectangle are small and concluding that the other two are also small. The only change is the relation between the sizes of the rectangle. The following is the main result of this section.

\begin{theorem}\label{pseudocwN}
Let $f$ be a $cw$-expansive homeomorphism with a self-similar hyperbolic $cw$-metric $D$. If $f$ has jointly continuous stable/unstable holonomies, then $f$ has pseudo-isometric joint stable/unstable holonomies. 
\end{theorem}

Before proceeding to the proof, we prove some consequences of the self-similar property for a hyperbolic $cw$-metrics. The following propositions are based in \cite{Ar}.

\begin{proposition}
Let $f$ be a $cw$-expansive homeomorphim with a self-similar $cw$-metric $D$. If $D(f(C_{[p,q]}))=\lambda D(C_{[p,q]})$, then $$D(f^k(C_{[p,q]}))=\lambda^k D(C_{[p,q]}),$$ for every $k>0$ such that $D(f^{k-1}(C_{[p,q]}))\leq \xi$.    
\end{proposition}
\begin{proof}
The result is trivial for $k=1$. Suppose that $k\geq 2 $ and that $D(f^{i}(C_{[p,q]}))=\lambda^{i} D(C_{[p,q]})$, for $i=0,1,...,k-1$. In particular, this implies that 
\begin{equation}\label{eqp} D(f^{k-1}(C_{[p,q]}))=\lambda D(f^{k-2}(C_{[p,q]})).
\end{equation}
Now, since by hypothesis we have $D(f^{k-1}(C_{[p,q]}))=\lambda^{k-1}D(C_{[p,q]})\leq \xi$, then the self-similarity implies $$\lambda D(f^{k-1}(C_{[p,q]}))= \max\{D(f^{k}(C_{[p,q]})),D(f^{k-2}(C_{[p,q]}))\}.$$
But we have $$ \lambda D(f^{k-1}(C_{[p,q]}))\neq D(f^{k-2}(C_{[p,q]})),$$
due to equation (\ref{eqp}) and this concludes the proof.
\end{proof}

\begin{proposition}
Let $f$ be a $cw$-expansive homeomorphim with a self-similar $cw$-metric $D$. If $\lambda D(C_{[p,q]})\neq D(f^{-1}(C_{[p,q]}))$, then $$D(f^k(C_{[p,q]}))=\lambda^k D(f(C_{[p,q]})),$$ for every $k>0$ such that $D(f^{k-1}(C_{[p,q]}))\leq \eps$.
\end{proposition}

\begin{proof}
Suppose that $k=1$.  If $\lambda D(C_{[p,q]})\neq D(f^{-1}(C_{[p,q]}))$, then by self-similarity we have $$\lambda D((C_{[p,q]}))= D(f(C_{[p,q]})).$$
Now we just need to apply the previous Proposition.
\end{proof}

\begin{proposition}\label{teocontrcw}
Let $f$ be a $cw$-expansive homeomorphim with a self-similar $cw$-metric $D$. \begin{itemize}
    \item If $C\subset \mathcal{C}^s_{\xi}$, then $D(f^k(C_{[p,q]}))=\lambda^{-k} D(C_{[p,q]})$ for every $k\geq0$.
    \item If $C\subset \mathcal{C}^u_{\xi}$, then $D(f^{-k}(C_{[p,q]}))=\lambda^{-k} D(C_{[p,q]})$ for every $k\geq0$.
    
    \end{itemize}
\end{proposition}

\begin{proof}
Suppose that $C\subset \mathcal{C}^s_{\xi}$ and $D(f^k(C_{[p,q]}))\neq \lambda^{-k} D(C_{[p,q]})$ for some $k>0$. Then there is some $m>0$ such that $D(f^{m-1}(C_{[p,q]}))\neq \lambda D(f^{m}(C_{[p,q]}))$, but by previous proposition we cannot have $C\subset \mathcal{C}^s_{\xi}$ and this is a contradiction.

\end{proof}

Now we are able to prove the main result of this section.

\begin{proof}[Proof of Theorem \ref{pseudocwN}]

Let $f$ be $cw$-hyperbolic homeomorphism with a $cw$-pseudo-metric $D$ and let $\eta>0$. Let $\lambda>1$ and $\xi>0$ be the expansion and expansiveness constants of $D$ respectively. Let $0<\epsilon\leq\xi$ such that $C^s_{\epsilon}(x)$ and $C^u_{\epsilon}(x)$ are non-trivial continuums for every $x\in M$ and fix $0<\delta<\epsilon$ such that $\pi^s_{x,y}$ and $\pi^u_{x,y}$ are well defined for every $x,y$ with respect to $\epsilon$ and $\delta$. Let $m_0>0$ be such that 
$$\frac{2}{\lambda^{m_0-1}-2}\leq \eta$$
Let $0<\beta\leq \frac{\epsilon}{\lambda^{m_0}}$ be given by the holonomy joint continuity, with $\alpha=\frac{\epsilon}{\lambda^{m_0}}$.
Take a continuum $C\subset C^s_{\epsilon}(x)$, $p,q\in C$, $p^*\in \pi_{x,y}^u(p)$, $C'\subset C^u(p)$ containing $p$ and $p^*$ such that $D(C_{(p,q)})\leq\beta$ and $D(C'_{(p,p^*)}) \leq \beta$. 
By joint continuity,  we can choose $C^*\subset C^s_{\epsilon}(y)$, $C^{**}\subset C^u_{\epsilon}(q)$ and $q^*\in C^*\cap C^{**}$ such that $$D(C^*_{(p^*,q^*)})\leq \frac{\epsilon}{\lambda^{m_0}} \textrm{ and } D(C^{**}_{(q,q^*)})\leq \frac{\epsilon}{\lambda^{m_0}}.$$
In order to conclude the joint pseudo-isometry property for $f$, we need to obtain the two estimates in Definition \ref{pseudo}. We now proceed to obtain the former.  The triangle inequality gives us the following estimates 
$$D(f^{n}(C_{(p,q)}))\leq D(f^{n}(C'_{(p,p^*)}))+D(f^{n}(C^*_{(p^*,q^*)}))+D(f^{n}(C^{**}_{(q^*,q)}))$$
\begin{center}and
\end{center}
$$D(f^{n}(C^*_{(p^*,q^*)}))\leq D(f^{n}(C'_{(p,p^*)}))+D(f^{n}(C_{(p,q)}))+D(f^{n}(C^{**}_{(q^*,q)})),$$
and combining above inequalities one has 
$$|D(f^{n}(C_{(p,q)}))-D(f^{n}(C^*_{(p^*,q^*)})|\leq D(f^{n}(C'_{(p,p^*)}))+D(f^{n}(C^{**}_{(q^*,q)}))$$
for every $n\in \Z$. 
This implies 
$$\max\{|D(f^{n}(C_{(p,q)}))-D(f^{n}(C^*_{(p^*,q^*)}))|,|D(f^{n}(C'_{(p,p^*)}))-D(f^{n}(C^{**}_{(q,q^*)}))|\}\leq\frac{2\epsilon}{\lambda^n}$$ for every $n\in \Z$.
By joint continuity, we can fix some $m>m_0$ such that 
$$\frac{\epsilon}{\lambda^{m-1}}\leq 2\max\{D(C{(p,q)}),D(C^*_{(p^*,q^*)}),D(C'{(p,p^*)}),D(C^{**}_{(q,q^*)})\}\leq\frac{\epsilon}{\lambda^{m}}.$$
Proposition \ref{teocontrcw} ensures 
\begin{eqnarray*}
|D(C_{(p,q)})-D(C^*_{(p^*,q^*)})|=\frac{1}{\lambda^{m}}|D(f^m(C_{(p,q)}))-D(f^m(C^*_{(p^*,q^*)}))|\leq\\\leq \frac{1}{\lambda^{m}} 2\max\{D(f^{m}(C_{(p,q)})),D(f^{m}(C^*_{(p^*,q^*)}))\}\leq\frac{2\epsilon}{\lambda^{2m}}
\end{eqnarray*}
and therefore 
\begin{equation}\label{pi}
|D(C_{(p,q)})-D(C^*_{(p^*,q^*)})|\leq\frac{2\max\{D(C_{(p,q)}),D(C^*_{(p^*,q^*)})\}}{\lambda^{m-1}},
\end{equation}
by the choice of $m$.
Suppose $\max\{D(C_{(p,q)}),D(C^*_{p^*,q^*})\}=D(C_{(p,q)})$. Then  (\ref{pi}) gives us 
$$\left|\frac{D(C^*_{(p^*,q^*)})}{D(C_{(p,q)})}-1\right| =\frac{|D(C_{(p,q)})-D(C^*_{(p^*,q^*)})|}{D(C_{(p,q)})}
  \leq\frac{2}{\lambda^{m-1}}<\frac{2}{\lambda^{m-1}-2}\leq\eta.$$
On the other hand, if  $\max\{D(C_{(p,q)}),D(C^*_{p^*,q^*})\}=D(C^*_{(p^*,q^*)})$, we have
$$\left|\frac{D(C_{(p,q)})}{D(C^*_{(p^*,q^*)})}-1\right| =\frac{|D(C_{(p,q)})-D(C^*_{(p^*,q^*)})|}{D(C^*_{(p^*,q^*)})}
  \leq\frac{2}{\lambda^{m-1}}.$$
Applying Lemma 4.2 in \cite{Ar} in the above inequality we obtain 
$$\left|\frac{D(C^*_{(p^*,q^*)})}{D(C_{(p,q)})}-1\right|\leq\frac{2}{\lambda^{m-1}-2}\leq \eta.$$
To obtain the second estimate in Definition \ref{pseudo}, we begin by obtaining, in an analogous way, the following estimate: 
$$|D(f^{n}(C'_{(p,p^*)}))-D(f^{n}(C^{**}_{(q,q^*)})|\leq D(f^{n}(C_{(p^*,q^*)}))+D(f^{n}(C^{*}_{(p^*,q^*)})).$$
Now by repeating the previous computations and by using the joint continuity we analogously obtain:     
$$\left|\frac{D(C^{**}_{(q,q^*)})}{D(C'_{(p,p^*)})}-1\right|\leq\frac{2}{\lambda^{m-1}-2}\leq\eta$$
and this  concludes the proof.
\end{proof}
  
\vspace{+0.6cm}

\section{Cw-hyperbolicity on Surfaces}

In this section we prove the joint continuity on stable/unstable holonomies in the case of cw-hyperbolic surface homeomorphisms $f\colon S\to S$. In this scenario we have more information on the structure of local stable/unstable continua that were proved in \cite{ACS}. The beginning of this section is devoted to explain the main contributions there. First, we observe that the assumption of a finite number of intersections between any pair of local stable/unstable continua is important for the results obtained in \cite{ACS}. This hypothesis is called cw$_F$-expansiveness and is defined as follows.

\begin{definition}\label{cwF}
A $cw$-expansive homeomorphism is said to be $cw_F$-expansive if there exists $c>0$ such that $$\#(C^s_c(x)\cap C^u_c(x))<\infty \quad \text{ for every } \quad x\in X.$$
Analogously, $f$ is said to be $cw_N$-expansive if there is $c>0$ such that $$\#(C^s_c(x)\cap C^u_c(x))\leq N \quad \text{ for every } \quad x\in X.$$   The $cw$-hyperbolic homeomorphisms that are $cw_F$-expansive (resp.\ $cw_N$-expansive) are called $cw_F$-hyperbolic (resp.\ $cw_N$-hyperbolic).
\end{definition}

This hypothesis is satisfied in the known examples of cw-hyperbolic surface homeomorphisms. The first result in \cite{ACS} is that for $cw_F$-hyperbolic surface homeomorphisms, there exists $\eps>0$ such that $C^s_{\eps}(x)$ and $C^u_{\eps}(x)$ are arcs, that is, homeomorphic to $[0,1]$ (see Proposition 2.12 in \cite{ACS}). In the case there is a homeomorphism $h\colon [0,1]\to C^s_{\eps}(x)$ such that $h(0)=x$, we say that $x$ is a spine. Let $\espinha(f)$ denote the set of all spines of $f$. The existence of spines is, indeed, an important feature of cw-hyperbolic surface homeomorphisms and is best illustrated in the pseudo-Anosov diffeomorphism of $\mathbb{S}^2$ (see Remark \ref{pseudo}). The second result in \cite{ACS} relates spines and bi-asymptotic sectors (discs bounded by a stable and an unstable arc) in a way that every regular sector contains a single spine and every spine is contained in a regular bi-asymptotic sector (see Proposition 3.6 and Lemma 3.7 in \cite{ACS}). The regular bi-asymptotic sectors are defined as being bounded by stable/unstable arcs with intersections pointing outward the sector (see Definition 3.1 in \cite{ACS} for a precise definition). The next figure is in \cite{ACS} and illustrates the difference between a regular and a non-regular sector.
 \begin{figure}[H]
        \begin{subfigure}[h]{0.49\linewidth}
            \centering
            \includegraphics[width=0.76\textwidth]{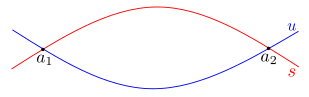}
            \caption{Regular sector}
            \label{fig:setorantigo}
        \end{subfigure}
        \hfill        
        \begin{subfigure}[h]{0.49\linewidth}
            \centering
            \includegraphics[width=0.76\textwidth]{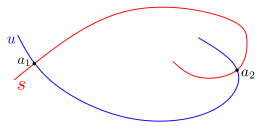}
            \caption{Non-regular sector}
            \label{fig:setorbugado1}
        \end{subfigure}
        \caption{}
    \end{figure}
A complete description of the structure of local stable and local unstable continua inside a regular bi-asymptotic sector is given in \cite{ACS}. There is a single spine in the interior of the sector and all other stable/unstable arcs intersect the boundary of the sector in two points turning around the spine (see Section 3 in \cite{ACS}). In the proof, a continuity for the variation of the stable/unstable arcs inside a regular bi-asymptotic sector is proved (see Lemma 3.5 in \cite{ACS}). We prove that inside a regular sector the joint continuity holds.

\begin{lemma}\label{continuitysector}
The holonomy maps $\pi_{x,y}^s$ and $\pi_{x,y}^u$ are jointly continuous
inside a regular bi-asymptotic sector.
\end{lemma}

\begin{proof}

We begin this proof parametrizing the whole regular sector $R$ using the stable arc $\gamma^s$ and an unstable arc $\gamma^u$ bounding it. Let $w$ be the spine in the interior of $R$, and let $z$ and $z'$ be the end points of $\gamma^s$ and $\gamma^u$. For each $p\in R$, let $C^u_R(p)$ and $C^s_R(p)$ be the connected components of $C^u(p)\cap R$ and $C^s(p)\cap R$ containing $p$, respectively. The curve $C=C^s_R(w)\cup C^u_R(w)$ splits $R$ in two regions $R_1$ and $R_2$. For any two points $p,q\in \inte(R)\setminus C$ the intersection $C^s_R(p)\cap C^u_R(q)$ contains exactly one point in $R_1$ and one point in $R_2$ (see Figure \ref{setor}). We denote by $C^u_{R_1}(p)$ and $C^u_{R_2}(p)$ the connected components of $C^u(p)\cap R_1$ and $C^u(p)\cap R_2$ containing $p$, respectively, and by $C^s_{R_1}(p)$ and $C^s_{R_2}(p)$ the connected components of $C^s(p)\cap R_1$ and $C^s(p)\cap R_2$ containing $p$, respectively.
\begin{figure}[h]
    \centering
    \includegraphics[scale=0.5]{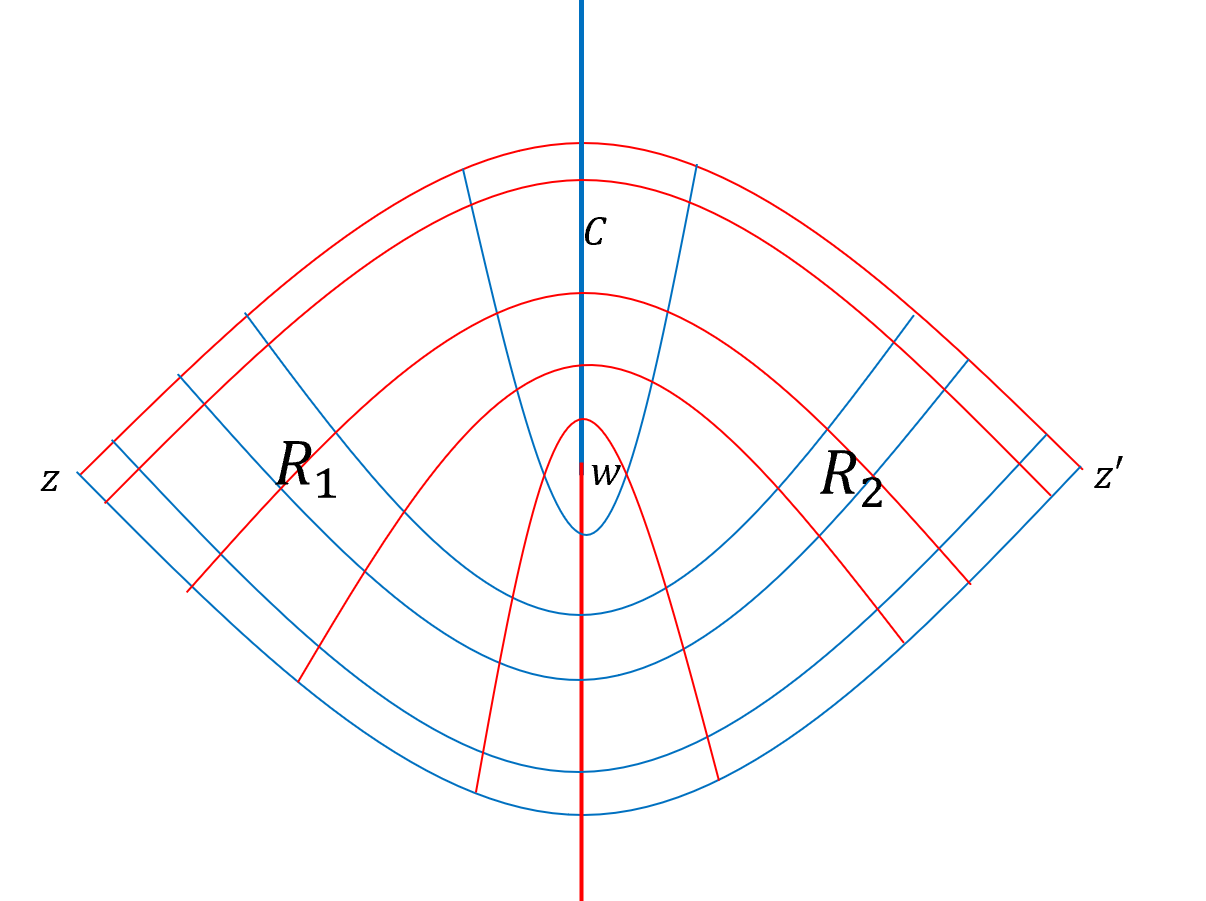}
    \caption{Portrait of the local stable and unstable sets inside a regular bi-asymptotic sector.}
    \label{setor}
\end{figure}
Reparametrizing, if necessary, we can assume that $\gamma^s,\gamma^u:[0,1]\to M$ satisfy 
$$\gamma^s(0)=\gamma^u(0)=z, \,\,\,\,\,\, \gamma^s(1)=\gamma^u(1)=z', \,\,\,\,\,\, \gamma^s\left(\frac{1}{2}\right)=\gamma^s\cap C,$$ 
$$\textrm{ and } \,\,\,\,\,\, \gamma^u\left(\frac{1}{2}\right)=\gamma^u\cap C.$$  
Define $f_1:[0,\frac{1}{2}]\times [0,\frac{1}{2}]\to R_1$ and $f_2:[\frac{1}{2},1]\times[\frac{1}{2},1]\to R_2$ by $$f_1(t,s)=C^u_{R_1}(\gamma^s(t))\cap C^s_{R_1}(\gamma^u(s)) \,\,\,\,\,\, \text{and}$$ 
$$f_2(t,s)=C^u_{R_2}(\gamma^s(t))\cap C^s_{R_2}(\gamma^u(s)).$$
We claim the maps $f_1$ and $f_2$ are continuous. Indeed, this is a consequence of the techniques developed in \cite{ACS}. We will briefly explain how to adapt their ideas to our context. To prove that $f_1$ is continuous it is enough to prove that it is continuous on each of its coordinates. 
Notice that if we fix $s_0$, then $f_1(t,s_0)$ maps $[0,\frac{1}{2}]$ injectively onto $C^s_{R_1}(\gamma^u(s_0))$.   In particular, it can be seen as a bijective map from $[0,\frac{1}{2}]$ onto $[0,\frac{1}{2}]$. Therefore, to prove that $f_1(t,s_0)$ is continuous, we just need to prove that it is monotonic. In \cite{ACS} it is proved that the sets of arcs $$\{C^s_{R}(p); p\in R\} \textrm{ and } \{C^u_{R}(p); p\in R\} $$ are totally ordered. With those orders, we can respectively induce total orderings on the sets of arcs $$\{C^s_{R_1}(p); p\in R_1\} \textrm{ and } \{C^u_{R_1}(p); p\in R_1\}.$$ Now the proof of the monotonicity of $f_1(t,s_0)$ follows exactly as the proof of Lemma 3.5 in \cite{ACS}. Analogously, we can prove that $f_1(t_0,s)$ is continuous for every $t_0\in [0,\frac{1}{2}]$ and  this implies the continuity of $f_1$. Analogously, we obtain the continuity of $f_2$ and the claim is proved.

We are finally able to conclude our proof. Let $x, y, p, q, p^*, C, C'$ be as in Definition \ref{cw-cont} and assume that they are contained in the interior of $R$. By the compatibility of the cw-metric and the diameter function, if $D(C_{(p,q)})$ and $D(C_{(p,p^*)})$ are small, then the distances $d(p,q)$ and $d(p,p^*)$ are also small. Now observe that $p,q,p^*$ can be seen as images of the maps $f_1$ and $f_2$. Next we split the choice of $q^*$ in two cases.

\vspace{0.1in}
\textbf{Case 1:} Suppose there is $i=1,2$ such that  $q,p^*$ are both in $R^i$. In this case, there are $t_q,s_q,t_{p^*},s_{p^*}\geq 0$ such that $q=f_i(t_q,s_q)$ and $p^*=f_i(t_{p^*},s_{p^*})$. Then we define $q^*=f_i(t_q,s_{q^*})$.

\vspace{0.1in}
\textbf{Case 2:} Suppose such that  $q$ and $p^*$ are not both in the same region. We will only prove the case when $q\in R_1$ and $p^*\in R_2$, since the reverse case is analogous.  In this case, there are $t_q,s_q\in [0,\frac{1}{2}]$ and $,t_{p^*},s_{p^*}\in [\frac{1}{2},1]$ such that $q=f_1(t_q,s_q)$ and $p^*=f_2(t_{p^*},s_{p^*})$. Let $q'=f^{-1}_2(f_1(t_q,\frac{1}{2}))$.  In this way, there is $t_{q'}\in [\frac{1}{2}, 1]$ such that $q'=f_2(t_{q'},\frac{1}{2})$. Then define $q^*=f_2(t_{q'},s_{p^*})$.

\vspace{0.1in}
\textbf{Case 3} Suppose that $p$ and $q$ are in distinct regions. In this case, proceed similarly as in the case 2, to choose $q^*$ in the same region as $q$ (see Figure \ref{cases}).

\begin{figure}[h]
    \centering
    \includegraphics[scale=0.6]{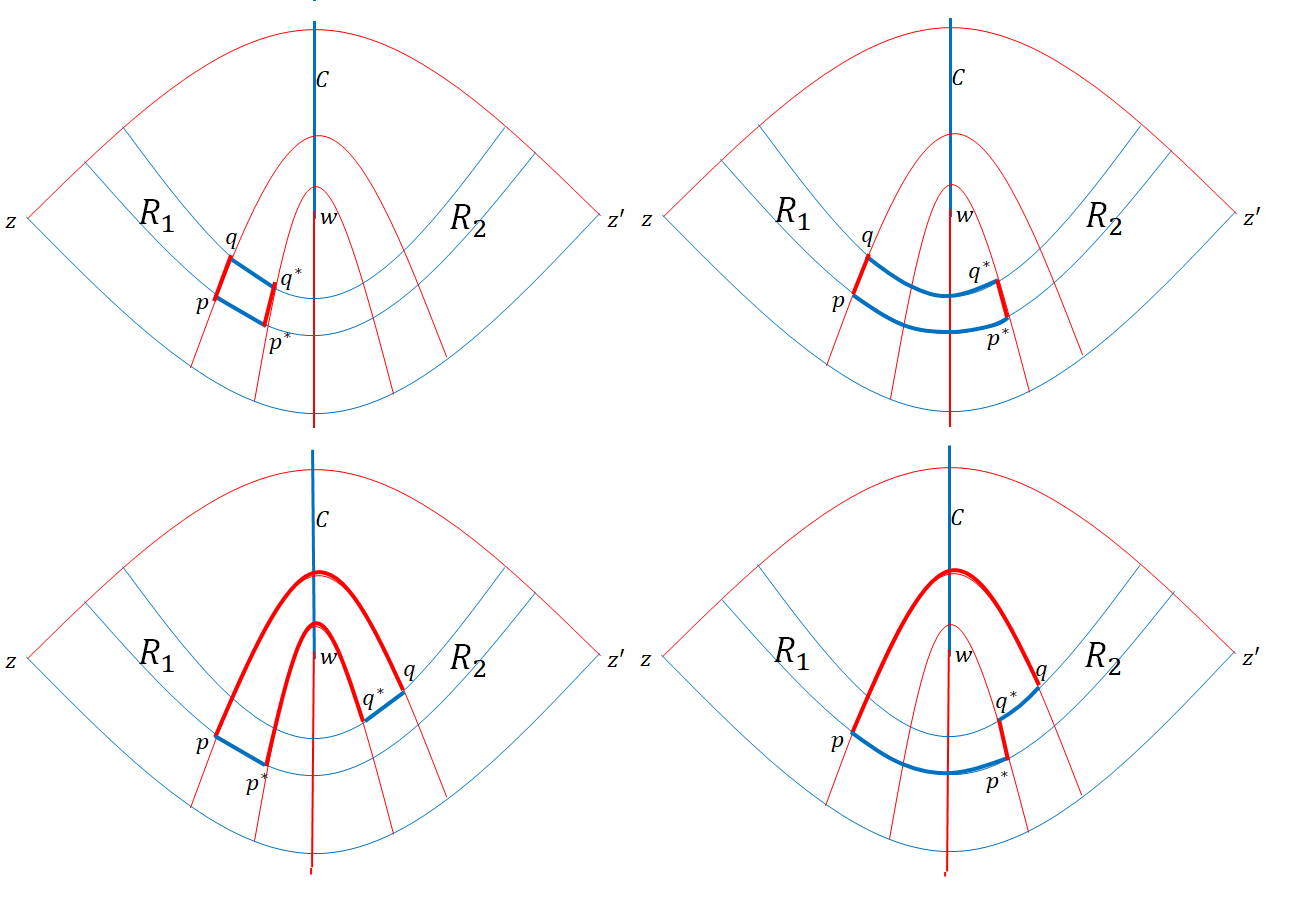}
    \caption{The Possibilities in Cases 1, 2 and 3.}
    \label{cases}
\end{figure}

By the choice of the parameters $t,s$ we conclude that the point  $q^*$  belongs simultaneously to  $\pi_{x,y}^u(z)$ and $\pi_{p,q}^s(p^*)$. 
Finally, the joint continuity follows from the continuity of the maps $f_1$ and $f_2$ and the compatibility of  the $cw$-metric $D$ with diameter function of $M$. 

\end{proof}

Having proved joint continuity inside regular bi-asymptotic sectors, we proceed to the proof of joint continuity as in Definition \ref{cw-cont}. The following lemma and proposition will be necessary for that. In the lemma we use the cw-local product structure to obtain intersections between local stable/unstable continua that become close to each other in a region far from the original point.

\begin{lemma}\label{LPS}
For each $p.q\in X$ and $\kappa\in(0,\frac{d(p,q)}{2})$, there exists $\eta\in(0,\kappa)$ satisfying: if $C^s_{\eps}(p)\cap B_{\eta}(q)\neq \emptyset$ and $C^u_{\eps}(p)\cap B_{\eta}(q)\neq \emptyset$, then there are $x\in C^s_{\eps}(p)\cap B_{\eta}(q)$ and $y\in C^u_{\eps}(p)\cap B_{\eta}(q)$ such that $C^s_{\kappa}(x)\cap C^u_{\kappa}(y)\neq\emptyset$.

\end{lemma}

\begin{proof}
Let $\eta'\in(0,\frac{\kappa}{2})$ be given by the $cw$-local product structure with respect to $\kappa$ and consider $\eta=\frac{\eta'}{2}$. Now, suppose $C^s_{\eps}(p)\cap B_{\eta}(q)\neq \emptyset$ and $C^u_{\eps}(p)\cap B_{\eta}(q)\neq \emptyset$ and consider $x\in C^s_{\eps}(p)\cap B_{\eta}(q)$ and $y\in C^u_{\eps}(p)\cap B_{\eta}(q)$. Thus, $d(x,y)\leq \eta'$ and the $cw$-local product structure implies $C^s_{\kappa}(x)\cap C^u_{\kappa}(y)\neq\emptyset$. 

\end{proof}

\begin{remark}
We observe that any point in the intersection $C^s_{\kappa}(x)\cap C^u_{\kappa}(y)$ obtained in the previous lemma is different from $p$. 
\end{remark}

\begin{proposition}\label{setormaior}
If all bi-asymptotic sectors of a cw$_F$-hyperbolic surface homeomorphism are regular, then every bi-asymptotic sector is contained in the interior of another bi-asymptotic sector.
\end{proposition}

\begin{proof}
Let $D$ be a bi-asymptotic sector of $f$ bounded by a stable arc $a^s$ and an unstable arc $a^u$, and let $a_1$ and $a_2$ be the end points of these arcs. First, we prove that $a_1$ and $a_2$ cannot be spines. Indeed, if $a_1$ is a spine, then we can choose long bi-asymptotic sectors arbitrarily close to $a_1$ as is done in Lemma 3.9 of \cite{ACS} (see also the proof of Theorem 1.1 and Figure 26 there) whose stable/unstable arcs $\gamma^s$ and $\gamma^u$ follow the stable/unstable arcs of $a_1$ till after the second intersection $a_2$ (see Figure \ref{bernardo1}). Note that $a_1$ and $a_2$ cannot be spines simultaneously.

\begin{figure}[h]
    \centering
    \includegraphics[scale=0.4]{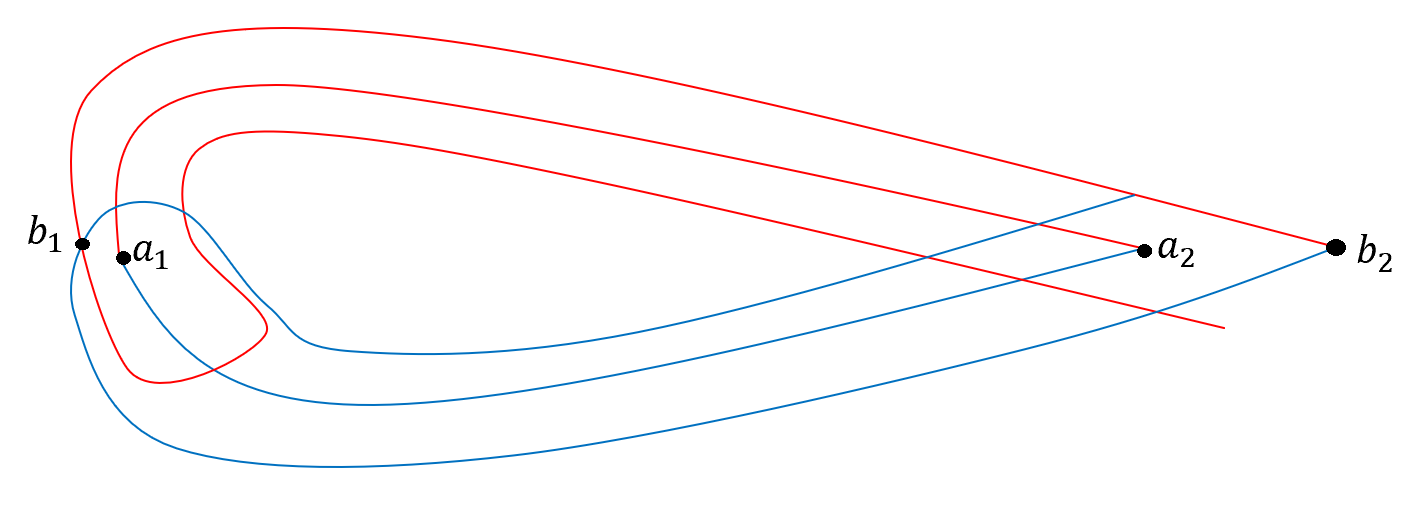}
    \caption{}
    \label{bernardo1}
\end{figure}

As indicated in this figure, we can choose an intersection $b_1\in\gamma^s\cap\gamma^u$ close to $a_1$ and another intersection $b_2\in\gamma^s\cap\gamma^u$ close to $a_2$, such that the sub-arcs of $\gamma^s$ and $\gamma^u$ connecting $b_1$ and $b_2$ form a bi-asymptotic sector with a non-regular intersection at $b_1$, contradicting the hypothesis. Thus, we can assume that both $a_1$ and $a_2$ are not spines, and, hence, there exist a stable arc $\alpha^s$ and an unstable arc $\alpha^u$ containing $a_1$ and contained in the exterior of $D$. As in the proof of Lemma 3.9 in \cite{ACS}, we can choose points $z_1\in\alpha^s$ and $z_2\in\alpha^u$ (arbitrarily close to $a_1$), an unstable arc $\beta^u$ containing $z_1$ and a stable arc $\beta^s$ containing $z_2$ such that both $\beta^u$ and $\beta^s$ separate a ball containing $D$, are contained in the exterior of $D$, and, forming a bi-asymptotic sector containing $D$ in its interior (see Figure \ref{bernardo2}). Lemma \ref{LPS} is important in this step to ensure an intersection between $\beta^s$ and $\beta^u$ close to $a_1$.

\begin{figure}[h]
    \centering
    \includegraphics[scale=0.4]{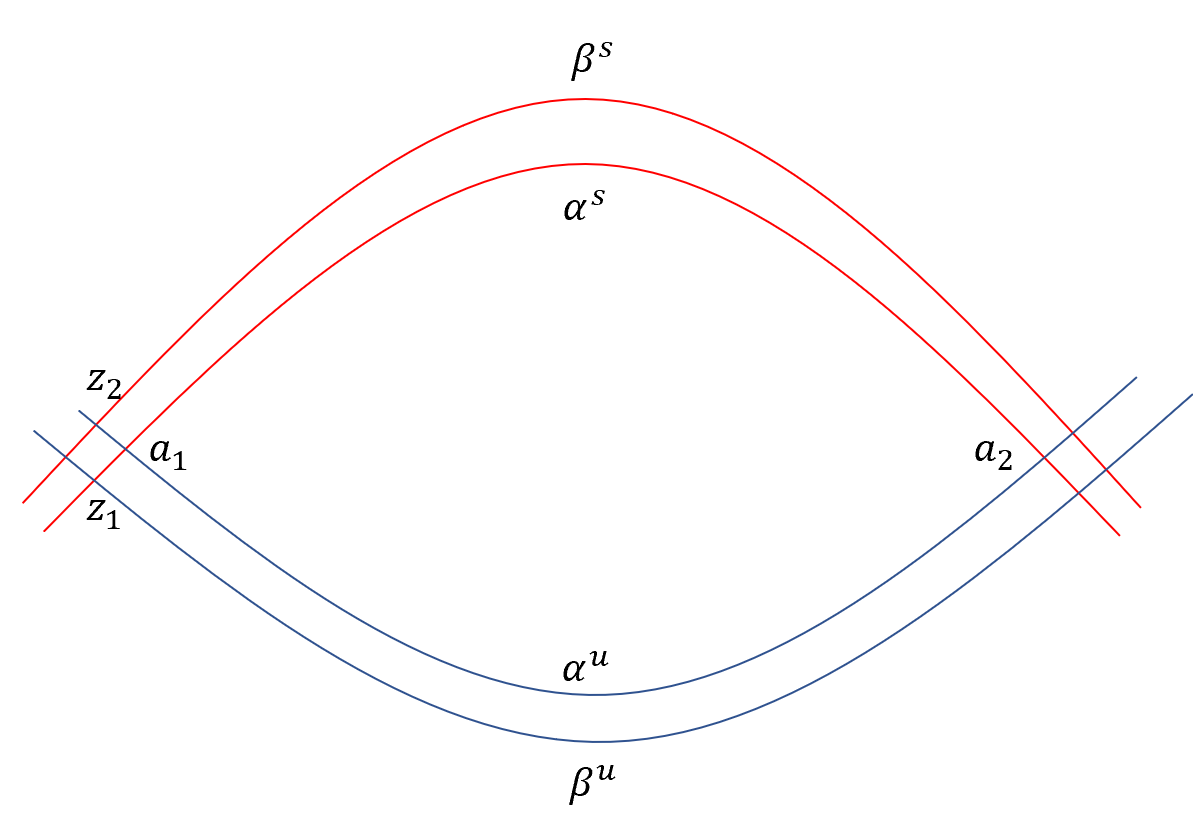}
    \caption{}
    \label{bernardo2}
\end{figure}
\end{proof}

\begin{proof}[Proof of Theorem \ref{cwfsurface}]
The proof goes by contradiction. Suppose all bi-asymptotic sectors of $f$ are regular and $f$ does not have jointly continuous stable/unstable holonomies. Then there exists $\alpha>0$ such that for each $n\in\N$ there are $x_n,y_n\in S$, $C_n\subset C^s_{\eps}(x_n)$, $p_n,q_n\in C_n$, $C'_n\subset C_{\eps}^u(p_n)$, and $p_n^*\in C'_n$ satisfying \begin{equation}\label{converg} D(C_{n(p_n,q_n)})\leq \frac{1}{n} \,\,\,\,\,\, \textrm{and} \,\,\,\,\,\, D(C'_{n(p_n,p^*_n)})\leq \frac{1}{n}
\end{equation} 
and such that for every continua $C^*\subset C^s_{\eps}(y_n)$, $C^{**}\subset C^u_{\eps}(q_n)$, and every $q_n^*\in \pi_{x_n,y_n}^s(q_n)$ we have either
\begin{equation}\label{bigholonomies}
    D(C^*_{n(p^*_n,q_n^*)})>\alpha \,\,\,\,\,\, \textrm{or} \,\,\,\,\,\, D(C^{**}_{n(q_n,q_n^*)})>\alpha.
\end{equation}
Next we will split our proof in two cases.

\textbf{Case 1:} Suppose $D(C^*_{n(p^*_n,q_n^*)})>\alpha$ for every $n\in\N$ and $D(C^{**}_{n(q_n,q_n^*)})\to 0$. Compactness of $S$ ensures, considering a subsequence if necessary, that $$p_n\to p, \,\,\, q_n\to q, \,\,\, p^*_n\to p^* \,\,\, \textrm{and} \,\,\, q^*_n\to q^*.$$ 
Since $D$ is compatible with the diameter function, inequalities (\ref{converg}) ensure that $p=q=p^*=q^*$. But since $D(C^*_{n(p^*_n,q_n^*)})>\alpha$ for every $n$, we obtain that $(C^*_{n(p^*_n,q_n^*)})_{n\in\N}$ accumulates in a non-trivial closed curve contained in $C^s_{\eps}(p^*)$. This contradicts the $cw$-expansiveness of $f$, since it generates a point with a stable set having non-empty interior. The case $D(C^*_{n(p^*_n,q_n^*)})\to 0$ and $D(C^{**}_{n(q_n,q_n^*)})>\alpha$ for every $n$ is analogous.

\textbf{Case 2:} Suppose $D(C^*_{n(p^*_n,q_n^*)})>\alpha$ for every $n>0$ but $D(C^{**}_{n(q_n,q_n^*)})$ does not converge to $0$. In this case, we can assume, possibly reducing $\alpha$ and taking sub-sequences, that 
$$D(C^*_{n(p^*_n,q_n^*)})>\alpha \,\,\,\,\,\, \text{and} \,\,\,\,\,\,  D(C^{**}_{n(q_n,q_n^*)})>\alpha \,\,\,\,\,\, \text{for every} \,\,\,\,\,\, n\in\N.$$
Compactness of $S$, compatibility of $D$ and $\diam$, and (\ref{converg}) ensure that $$p_n,q_n,p^*_n\to p.$$
Inequalities (\ref{bigholonomies}) ensure that $(C^*_{n(p^*_n,q_n^*)})_{n\in\N}$ and $(C^{**}_{n(q_n,q_n^*)})_{n\in\N}$ accumulate in a pair of non-trivial stable/unstable continua $C^s$ and $C^u$ intersecting in at least two distinct points. Recall that they intersect in at most a finite number of points since $f$ is cw$_F$-expansive. This creates one bi-asymptotic sector $D$ bounded by sub-arcs of $C^s$ and $C^u$ containing $p$, which by assumption is regular. Proposition \ref{setormaior} ensures the existence of a bi-asymptotic sector $D'$ containing $D$ in its interior. Since $p_n,q_n,p_n^*\to p$, and $p\in D$, it follows that $p_n,q_n,p^*_n\in \inte(D')$, and also $C_{n(p_n,q_n)}, C_{n(p_n,p_n^*)}\subset\inte(D')$, if $n$ is sufficiently big. Therefore, Lemma \ref{continuitysector} ensures the existence of arcs $\alpha^s_n$ and $\alpha^u_n$ contained in $D'$, containing $p_n^*$ and $q_n$, respectively, and $q_n^{**}\in \alpha^s_n\cap\alpha^u_n$ satisfying $\diam(\alpha^s_n)\to0$ and $\diam(\alpha^u_n)\to0$. This contradicts the hypothesis done at the beginning of the proof that inequalities \ref{bigholonomies} hold for every choice of $q_n^*\in \pi_{x_n,y_n}^s(q_n)$. This concludes the proof.

\end{proof}

We end this section proving that the hypothesis of regularity on all bi-asymptotic sectors holds for a class of cw-hyperbolic homeomorphisms discussed in \cite{ACS}.

\begin{proposition}\label{allregular}
If a cw$_F$-hyperbolic surface homeomorphism admits only a finite number of spines, then all its bi-asymptotic sectors are regular.
\end{proposition}

\begin{proof}
If $f\colon S\to S$ is a cw$_F$-hyperbolic surface homeomorphism admitting only a finite number of spines, then $f$ is cw$_2$-hyperbolic (see \cite{ACS}*{Theorem1.1}). This means that any pair of local stable/unstable continua can only intersect in two distinct points. The proof is the observation that the existence of a non-regular bi-asymptotic sector implies the existence of at least three distinct intersections between pairs of local stable/unstable continua. Let $D$ be a bi-asymptotic sector bounded by a stable arc $a^s$ and an unstable arc $a^u$ intersecting at $a_1$ and $a_2$, and assume that the intersection in $a_2$ is not regular as in Figure \ref{fig:setorbugado1}. The arcs $C^s_D(a_1)\setminus\{a^s\}$ and $C^u_D(a_1)\setminus\{a^u\}$ either separate $D$ or contain a spine in $\inte(D)$ (see Lemma 2.13 in \cite{ACS}). The cw$_2$-expansiveness ensures the first case is not possible (see Figure 21 in \cite{ACS}), so both arcs $C^s_D(a_1)\setminus\{a^s\}$ and $C^u_D(a_1)\setminus\{a^u\}$ end in a spine in $\inte(D)$. The cw$_2$-expansiveness also ensure that these spines are distinct. The cw-local product structure ensure that the stable and unstable continua of these spines intersect and, hence, long bi-asymptotic sectors around the spines must intersect in at least four distinct points (see the proof of Theorem 1.1 and Figure 26 in \cite{ACS}). This contradicts cw$_2$-expansiveness and finishes the proof.
\end{proof}

\vspace{+0.6cm}

\section{Isometric local stable/unstable holonomies and transitivity}
In this section, we prove transitivity of $cw$-hyperbolic homeomorphism assuming isometric stable/unstable holonomies. In \cite{Ar} isometric stable/unstable holonomies are used to prove transitivity for topologically hyperbolic homeomorphisms. We prove a similar result in the case of cw-hyperbolic homeomorphisms. In this case, the proof is, indeed, an adaptation of the techniques of \cite{Ar}. We note that the joint holonomies will not be necessary in the proof.

\begin{definition}[Isometric stable/unstable holonomies]
We say that $\pi^s_{x,y}$ is an \emph{isometric local stable holonomy} if for every continuum $C\subset C^s_{\eps}(x)$ and $p,q\in C$, there are $C^*\subset C^s_{\eps}(y)$, $p^*\in \pi^s_{x,y}(p)$ and $q^*\in \pi^s_{x,y}(q)$ such that 
$$D(C^*_{(p^*,q^*)})=D(C_{(p,q)}).$$ 
Analogously, we can define what means $\pi^u_{x,y}$ to be a \emph{isometric local unstable holonomy}. We say that $f$ has isometric stable (unstable) holonomies if all local stable (unstable) holonomies of $f$ are isometric.
\end{definition}
As noticed in \cite{Ar}, this property does not necessarily hold even if $f$ is topologically hyperbolic, but assuming this as hypothesis we can prove the following result.

\begin{theorem}\label{trans}
If $f\colon X\to X$ is a $cw$-hyperbolic homeomorphism that has isometric stable (or unstable) holonomies, then $f$ is transitive.
\end{theorem}
 
Before proving this theorem, we need to prove some properties about the spectral decomposition of cw-hyperbolic homeomorphisms that were not explored in \cite{ACCV3}. The following was proved in \cite{ACCV3}:
 
\begin{theorem}[Spectral Decomposition]\label{spec}
If $f$ is a $cw$-hyperbolic homeomorphism, then
\begin{enumerate}
     \item $f$ admits only a finite number of chain-recurrent classes $C_1,\dots,C_k$,
     \item each class either is expansive or admits arbitrarily small totally disconnected topological semi-horseshoes,
     \item each chain recurrent classes is transitive and admits a periodic decomposition into elementary sets $D_1,...,D_n$,
     \item $f^n$ restricted to each elementary set is topologically mixing.
 \end{enumerate}
 \end{theorem}

In the proof of Theorem \ref{trans} we need the following additional result:

\begin{proposition}\label{ar}
If $f$ is a $cw$-hyperbolic homeomorphism, then either $f$ is transitive, or there exist an attractor and a reppelor chain-recurrent classes.
\end{proposition}

We first recall the definitions needed in the proof. We say that $f\colon X\to X$ is \emph{transitive}, if for any pair $U,V\subset X$ of non-empty open subsets, there exists $n\in\N$ such that $f^n(U)\cap V\neq\emptyset$. A point $x\in X$ is called \emph{chain-recurrent} if for each $\eps>0$ there exists a non-trivial finite $\eps$-pseudo orbit starting and ending at $x$. An $\eps$-pseudo-orbit is a set of points $(x_k)_{k\in I}$ satisfying $d(f(x_k),x_{k+1})\leq\eps$ for every $k\in I$. The set of all chain-recurrent points is called the \textit{chain recurrent set} and is denoted by $CR(f)$. This set can be split into disjoint, compact and invariant subsets, called the \emph{chain-recurrent classes}. The \emph{chain-recurrent class} of a point $x\in X$ is the set of all points $y\in X$ such that for each $\eps>0$ there exist a periodic $\eps$-pseudo orbit containing both $x$ and $y$. If $f$ is transitive, then the whole space $X$ is a chain recurrent class. A chain-recurrent class $C$ is an attractor if $W^u(C)=C$, where $$W^u(C)=\{x\in X; \,\,\, d(f^k(x),C)\to 0 \,\,\,\,\,\, \text{when} \,\,\,\,\,\, k\to-\infty\},$$ and is a reppelor if $W^s(C)=C$, where $$W^s(C)=\{x\in X; \,\,\, d(f^k(x),C)\to 0 \,\,\,\,\,\, \text{when} \,\,\,\,\,\, k\to\infty\}.$$

\begin{proof}[Proof of Proposition \ref{ar}]
If $f$ is not transitive, then there are distinct chain-recurrent classes $C_1,\dots,C_k$. Recall that $f$ has the L-shadowing property \cite{ACCV3} and, hence, Proposition 2.12 in \cite{ACT} ensures that $$X=\bigcup_{i=1}^kW^s(C_i)=\bigcup_{i=1}^kW^u(C_i).$$
We define an order $<$ on the set $\{C_1,\dots,C_k\}$ as follows: we say that $C_i<C_j$ if
$$(W^u(C_i)\setminus\{C_i\})\cap(W^s(C_j)\setminus\{C_j\})\neq\emptyset.$$ Any maximal class $C$ with respect to this order is an attractor. Indeed, if $z\in W^u(C)\setminus\{C\}$, then above equalities ensure that $z\in W^s(C_i)$ for some class $C_i$, and since $z$ is non-wandering, it follows that
$$z\in(W^u(C)\setminus\{C\})\cap(W^s(C_i)\setminus\{C_i\}),$$ that is, $C<C_i$, contradicting the maximality of $C$. Analogously, any minimal class with respect to this order is a reppeller.
\end{proof}

\begin{proof}[Proof of Theorem \ref{trans}]
Assume that local stable holonomies are isometric and that $f$ is not transitive. Proposition \ref{ar} ensures the existence of attractors and reppelors classes of $f$. Consider a reppelor $R$ and an attractor $A$ such that $R<A$ and let $x\in X$ be such that $f^{-n}(x)\to R$ and $f^n(x)\to A$ when $n\to\infty$. We claim that there exist $y\in R$, $z\in A$, $p\in X$ and $n\in\N$ such that 
$$f^{-n}(p)\in C^u_{\eps}(y) \,\,\,\,\,\, \text{and} \,\,\,\,\,\, f^{n}(p)\in C^s_{\eps}(z).$$
Let $\delta\in(0,\frac{\eps}{2})$ be given by the cw-local product structure for $\frac{\eps}{2}$ and let $n\in\N$ be such that $$\diam(f^{2n}(C))<\frac{\delta}{2} \,\,\,\,\,\, \text{for every} \,\,\,\,\,\, C\in\mathcal{C}^s_{\eps},$$
$$d(f^{-n}(x),R)<\frac{\delta}{2} \,\,\,\,\,\, \text{and} \,\,\,\,\,\, d(f^n(x),A)<\frac{\delta}{2}.$$ Thus, there exist $y\in R$, $z\in A$, $$q\in C^u_{\frac{\eps}{2}}(y)\cap C^s_{\frac{\eps}{2}}(f^{-n}(x)) \,\,\,\,\,\, \text{and} \,\,\,\,\,\, q'\in C^u_{\eps}(f^{2n}(q))\cap C^s_{\eps}(z),$$ so $p=f^{-n}(q')$ satisfies $f^{-n}(p)\in C^u_{\eps}(y)$ and $f^{n}(p)\in C^s_{\eps}(z)$.
 
Note that $\gamma=f^{2n}(C^u_{\eps}(y))$ is an unstable continuum containing both $f^{2n}(y)$ and $f^{n}(p)$, and that $f^{2n}(y)\in R$, while $f^{n}(p)$ is $\frac{\delta}{2}$-close to $A$. Consider $q_0,\dots,q_k\in\gamma$ such that $q_0=f^n(p)$, $q_k=f^{2n}(y)$ and $d(q_i,q_{i+1})\leq \delta$ for every $i\in\{1,...,k\}$. Let $C\subset C^s_{\eps}(z)$ be such that $f^n(p),z\in C$. Since local stable holonomies are isometric, there exist $C^*\subset C^s_{\eps}(q_1)$ and $p_1\in C^s_{\eps}(q_1)$ such that $D(C^*_{(q_1,p_1)})=D(C_{(f^n(p),z)})$. Also, there exists $C^{**}\subset C^u_{\eps}(z)\subset A$ such that $p_1,z\in C^{**}$.
Inductively, we obtain continua $(C^*_i)_{i=1}^k$, $(C^{**}_i)_{i=1}^k$, and points $(p_i)_{i=1}^k$ such that 
$$C^*_i\subset C^s_{\eps}(q_i), \,\,\,\,\,\, C^{**}_i\subset C^u_{\eps}(p_{i-1})\subset A, \,\,\,\,\,\, q_i,p_i\in C^*_i \,\,\,\,\,\, p_i,p_{i-1}\in C^{**}_i, \,\,\,\,\,\, \text{and}$$
$$D(C^*_{i(q_i,p_i)})=D(C_{(f^n(p),z)}) \,\,\,\,\,\, \text{for every} \,\,\,\,\,\, i\in\{1,\dots,n\}.$$
Thus, $p_k\in C^{**}_k\subset A$ and $p_k\in C^*_k\subset C^s_{\eps}(q_k)=C^s_{\eps}(f^{2n}(y))\subset R$, contradicting $A\cap R=\emptyset$. This proves that $f$ is transitive. The case where local unstable holonomies are isometric is analogous, the only change is considering the curve $\gamma$ as $f^{-2n}(C^s_{\eps}(z))$ and use the unstable holonomies along $\gamma$.

\end{proof}

\vspace{+0.8cm}

\hspace{-0.4cm}\textbf{Acknowledgments.}
The authors thank Alfonso Artigue for discussions about his work \cite{Ar} during the preparation of this work. Bernardo Carvalho was supported by Progetto di Eccellenza MatMod@TOV grant number PRIN 2017S35EHN, by CNPq grant number 405916/2018-3, and Elias Rego was also supported by Fapemig grant number APQ-00036-22.

\vspace{1.0cm}

{\em B. Carvalho}
\vspace{0.2cm}

\noindent

Dipartimento di Matematica,

Università degli Studi di Roma Tor Vergata

Via Cracovia n.50 - 00133

Roma - RM, Italy

\email{mldbnr01@uniroma2.it}

\vspace{1.0cm}

{\em E. Rego}
\vspace{0.2cm}

Department of Mathematics

Southern University of Science and Technology

Xueyuan Blvd n. 1088, Nanshan 

Shenzhen -Guangdong, China 

\email{rego@sustechedu.cn}

\noindent

\vspace{0.2cm}

\end{document}